\documentclass[11pt]{article}

\usepackage{geometry}

\usepackage{amsmath}
\usepackage{amssymb}
\usepackage{amsthm}
\usepackage{graphicx}
\usepackage{authblk}
\usepackage{color}
\usepackage{setspace}

\usepackage{paralist}

\usepackage{hyperref}
\hypersetup{
  colorlinks=false,
  linkcolor=cyan,
  citecolor=green,
  filecolor=blue,
  urlcolor=red,
}
\hypersetup{hidelinks}

\allowdisplaybreaks[4]

\makeatletter
\@addtoreset{equation}{section}
\makeatother

\newtheorem{lemma}{Lemma}[section]
\newtheorem{proposition}{Proposition}[section]
\makeatletter
\renewcommand{\section}{\@startsection{section}{1}{0mm}
 {0.7\baselineskip}{0.5\baselineskip}{\Large\bf\leftline}}
\makeatother

\newtheorem{theorem}{Theorem}[section]

\theoremstyle{definition}

\newtheorem{remark}{Remark}

\begin{document}
\title{\bf Large-time behavior of solutions to the outflow problem for the compressible Navier-Stokes-Maxwell equations} \vskip 0.5cm

\author{Huancheng Yao, Changjiang Zhu\thanks{Corresponding author.
\authorcr E-mail addresses: mayaohch@mail.scut.edu.cn (Yao), machjzhu@scut.edu.cn (Zhu).}}
\affil{ \normalsize  School of Mathematics, South China University of Technology, \authorcr Guangzhou 510641, P.R. China }

\date{}
\maketitle

\begin{spacing}{1.24}

\begin{abstract}

We investigate the large-time behavior of solutions toward the combination of the boundary layer and 3-rarefaction waves to the outflow problem for the compressible non-isentropic Navier-Stokes equations coupling with the Maxwell equations through the Lorentz force (called the Navier-Stokes-Maxwell equations) on the half line $ \mathbb{R}_+ $.
It includes the electrodynamic effects into the dissipative structure of the hyperbolic-parabolic system and turns out to be more complicated than that in the simpler compressible Navier-Stokes equations.
We prove that this typical composite wave pattern is time-asymptotically stable with the composite boundary condition of the electromagnetic fields,
under some smallness conditions and the assumption that the dielectric constant is bounded.
This can be viewed as the first result about the nonlinear stability of the combination of two different wave patterns for the IBVP of the non-isentropic Navier-Stokes-Maxwell equations.

\vspace{3mm}


\vspace{2mm}


\end{abstract}

\begin{spacing}{0.90}
\tableofcontents
\end{spacing}

\section{Introduction}

\noindent Plasma dynamics is a field of studying flow problems of electrically conducting fluids.
A complete analysis in this broad field includes the study of the gasdynamic field, the electromagnetic fields and the radiation field simultaneously in \cite{Pai1962}.
In this paper, we consider the motion of an electrically conducting fluid in the presence of electric field and  magnetic field.
At the macroscopic level, the flow of this electrically conducting fluid such as the movement in the electromagnetic fields generated by itself is described by hydrodynamics equations,
for example the compressible Navier-Stokes equations.
Since the dynamic motion of the fluid and the electromagnetic fields couple strongly,
the governing system in the non-isentropic case is derived from fluid mechanics with appropriate modifications to take account of the electromagnetic effects,
which consists of the laws of conservation of mass, momentum and energy, Maxwell's law, and the law of conservation of electric charge (see \cite{Imai}, \cite{Kawashima1}).
In this paper, we shall restrict ourselves to the one-dimensional motion (see \cite{fanhu}, \cite{yaozhu2021}) on the half line $ \mathbb{R}_+ $:
\begin{equation}\label{yuanfangcheng}
\left\{
\begin{aligned}
   &\rho_t + (\rho u)_x = 0,   \\
   &\rho (u_t + uu_x ) + p_x = \mu u_{xx} - (E + u b)b,   \\
   &\frac{R}{\gamma-1} \rho (\theta_t  + u \theta_x ) + p u_x = \mu u_x^2  +  \kappa \theta_{xx}  + (E + u b)^2,   \\
   &\varepsilon E_t  - b_x + E + u b =0,  \\
   &b_t - E_x =0,
\end{aligned}
\right.
\end{equation}
where $(x, t)\in\mathbb{R}_{+}\times\mathbb{R}_{+}$.
The detailed mathematical derivation of system \eqref{yuanfangcheng} will be given in Appendix.
Here, $\rho(x,t)>0$ denotes the mass density;
$u(x,t)$ is the fluid velocity;
$ \theta(x,t) > 0 $ is the absolute temperature;
$E(x,t)$ and $b(x,t)$ denote the electric field and the magnetic field, respectively.
The pressure $ p $ is expressed by the equations of states.
For the sake of simplicity, we will focus on only polytropic fluids throughout this paper, namely
\begin{equation}\label{zhuangtaifangcheng-p}
p = R \rho \theta = A \rho^{\gamma} \mathrm{exp}\left( \frac{\gamma-1}{R} s \right),
\end{equation}
where $ s $ is the entropy.
The parameters in the above equations, respectively,
$ R >0 $ is the gas constant and $ \gamma>1 $ is the adiabatic exponent;
$\mu >0 $ in $\eqref{yuanfangcheng}_2$ and $\eqref{yuanfangcheng}_3$ is the viscosity coefficient of the fluid;
the heat conductivity coefficient $\kappa$ in $\eqref{yuanfangcheng}_3$ is assumed to be a positive constant.
Moreover, $\varepsilon>0$ in $\eqref{yuanfangcheng}_4$ denotes the dielectric constant.

The system \eqref{yuanfangcheng} is obtained from the Navier-Stokes equations coupling with the Maxwell equations through the Lorentz force.
Thus it is usually called the Navier-Stokes-Maxwell equations.
Notice that the same terminology was used by Masmoudi in \cite{masmoudi-jmpa-2010} and Duan in \cite{duan-aa-2012} but for the different models.
In this paper, we consider the initial-boundary value problem for the compressible non-isentropic Navier-Stokes-Maxwell equations on a half line.
The initial data for the system \eqref{yuanfangcheng} is given by
\begin{equation}\label{yuanfangcheng-chuzhi}
  (\rho, u, \theta, E, b)(x,0) = (\rho_0, u_0, \theta_0, E_0, b_0)(x), \quad  \inf_{x \in \mathbb{R}_+} \rho_0 (x) >0, \quad  \inf_{x \in \mathbb{R}_+} \theta_0 (x) >0 .
\end{equation}
We assume that the initial data in the far field $ x =+\infty $ is constant, namely
\begin{equation*}
  \lim_{x\rightarrow +\infty} (\rho_0, u_0, \theta_0, E_0, b_0)(x) = (\rho_+, u_+, \theta_+, E_+, b_+),
\end{equation*}
and the boundary data for $ u $ and $ \theta $ at $x = 0$ is given by the following constants
\begin{equation*}
  ( u, \theta )(0,t) = ( u_-, \theta_- ), \quad \forall \; t \geq 0,
\end{equation*}
where $ \rho_+ > 0 $, $ \theta_\pm > 0 $ and $ u_- < 0 $.
The following compatibility conditions hold as well
\begin{equation*}
  u_0(0)= u_-,  \quad \theta_0(0) = \theta_-.
\end{equation*}
In particular, we suppose the boundary values for $ E $ and $ b $ satisfy the following condition:
\begin{equation}\label{bdycons-diancichang}
 \left( \sqrt{\varepsilon} E - b \right)(0,t)  =  0 .
\end{equation}

Actually, setting $ V = \left( \varepsilon E , \; b  \right)^T  $,
equations $\eqref{yuanfangcheng}_4 $ and $\eqref{yuanfangcheng}_5 $ can be transformed as the following form
\begin{align}\label{tiboundary-yuanfangcheng}
  V_t  +  A V_x  +  B V =0, \quad \mathrm{for} \quad
  A =
  \left(\begin{array}{cc}
  \displaystyle  0      \;   &     \displaystyle  -1       \\[2mm]
  \displaystyle    -1/\varepsilon  \;    &     \displaystyle 0
  \end{array}\right),
  \quad        B =
              \left(\begin{array}{cc}
              \displaystyle  1/\varepsilon     \;    &   \displaystyle  u       \\[2mm]
              \displaystyle  0    \;   &    \displaystyle  0
              \end{array}\right).
\end{align}
Here Matrix $ A $ has two eigenvalues $ \lambda_1 = 1/\sqrt{\varepsilon} $ and $ \lambda_2 = - 1/\sqrt{\varepsilon} $.
Direct calculations show that the pair of Riemann invariants $\left\{ W_1,  \;  W_2 \right\}$ associated with the eigenvalues $ \lambda_1 $ and $ \lambda_2 $ can be taken as
\begin{equation}\label{tiboundary-riemann-inv}
  \left\{ W_1,  \;  W_2 \right\}    =    \frac{\sqrt{\varepsilon}}{2} \left\{\sqrt{\varepsilon} E - b,  \;  \sqrt{\varepsilon} E + b \right\} .
\end{equation}
Using this pair of Riemann invariants, we can diagonalize the equations in \eqref{tiboundary-yuanfangcheng} as
\begin{equation*}
  W_t + \Lambda W_x + D W = 0,
\end{equation*}
where $ W : = \left( W_1 , \;   W_2  \right)^T $ and $ \Lambda := \mathrm{diag}(\lambda_1, \; \lambda_2) $.
The basic theory of hyperbolic systems of conservation laws (for example, see \cite{Smoller}) shows that
we must specify the boundary value $ W_1(0,t) $ since $ \lambda_1 > 0 $;
and $ W_2(0,t) $ is determined by the initial data since $ \lambda_2 < 0 $.
Now each $ W_i\;(i=1\;\mathrm{or}\; 2) $ is a linear combination of
the $ V_i\;(i = 1\;\mathrm{and}\;2) $, so we must specify only one condition on the components
of $ V $ at the boundary $ x = 0 $.
Due to \eqref{tiboundary-riemann-inv},
we specify the boundary value on $ \sqrt{\varepsilon} E - b $ satisfying \eqref{bdycons-diancichang}.
It should be pointed out that this type of boundary condition has ever been mentioned by Chen-Jerome-Wang in \cite{chen-wang-2000}.


Let us recall some known results about the Navier-Stokes-Maxwell equations.
There have been some research on the existence and large-time behavior of solutions,
and the vanishing dielectric constant limit problem to the compressible Navier-Stokes-Maxwell equations.
In \cite{Kawashima2} and \cite{Kawashima3}, Kawashima and Shizuta established the global existence of smooth solutions for small data and studied its zero dielectric constant limit in the whole space $ \mathbb{R}^2 $.
Li and Mu \cite{li-mu-jmaa-2014} studied the low Mach number limit problem for the solution of the full compressible Navier-Stokes-Maxwell equations converging to that of the incompressible system in $ \mathbb{R}^3 $.
Later, Jiang and Li in \cite{jiangsong1} studied the vanishing dielectric constant limit and obtained the convergence of the 3-D Navier-Stokes-Maxwell equations to the full compressible magnetohydrodynamic equations in the torus $ \mathbb{T}^3 $.
Recently, Xu in \cite{xux} studied the large-time behavior of the classical solution toward some given constant states and obtained the time-decay estimates in the whole space $ \mathbb{R}^3 $ with small initial perturbation in $ H^3\cap L^1 $.
For the one-dimensional non-isentropic model, Fan and Hu in \cite{fanhu} obtained the uniform estimates with respect to the dielectric constant and the global-in-time existence in a bounded interval without vacuum.
Furthermore, Fan and Ou in \cite{fanou} considered the one-dimensional full equations for a thermo-radiative electromagnetic fluid in a form similar to that in \eqref{yuanfangcheng};
and established the similar result to \cite{fanhu}.

However, for the one-dimensional compressible Navier-Stokes-Maxwell equations,
there are few results about the large-time behavior of the solution toward some non-constant states, especially wave patterns. 
To the authors' best knowledge, there are only four relevant results.
To the Cauchy problem, Luo-Yao-Zhu \cite{luoyaozhu2021} and Yao-Zhu \cite{yaozhu2021} established the stability of rarefaction wave for the compressible isentropic and non-isentropic Navier-Stokes-Maxwell equations under suitable smallness conditions, respectively.
Huang-Liu in \cite{huangyt} consider the stability of rarefaction wave for a macroscopic model derived from the Vlasov-Maxwell-Boltzmann system,
in which the model they consider is obviously different from this in our paper,
except for the similar dissipative term $E + ub$.
Recently, Yao-Zhu in \cite{yaozhu2021-fuhebo} study the asymptotic stability of the superposition of viscous contact wave with rarefaction waves for the compressible Navier-Stokes-Maxwell equations,
which is the first result on the combination of two different wave patterns of this complex coupled model.
But for the large-time behavior of solutions to an IBVP of the non-isentropic Navier-Stokes-Maxwell equations,
as far as we know there are still few results.
Here, we will partly give a positive answer for this important problem.

In fact, equations \eqref{yuanfangcheng} reduce to the classical Navier-Stokes equations if we ignore the effects of the electromagnetic fields.
Motivated by the relationship between the Navier-Stokes-Maxwell equations and Navier-Stokes equations,
we temporarily assume that $ E_+ = b_+  = 0 $, namely, the initial-boundary values satisfying
\begin{equation}\label{yuanfangcheng-chuzhi-shuju}
  \lim_{x\rightarrow +\infty} (\rho_0, u_0, \theta_0, E_0, b_0)(x) = (\rho_+, u_+, \theta_+, 0, 0),
\end{equation}
\begin{equation}\label{yuanfangcheng-bianzhi}
  (u, \theta, \sqrt{\varepsilon} E - b)(0,t) = (u_-, \theta_-, 0), \quad \forall \; t \geq 0 ,
\end{equation}
and can consider the large-time behavior of solutions to the outflow
problem \eqref{yuanfangcheng}-\eqref{yuanfangcheng-chuzhi} and \eqref{yuanfangcheng-chuzhi-shuju}-\eqref{yuanfangcheng-bianzhi} in the setting of $ E(x,t) = b(x,t) = 0 $.
Then the above outflow problem is reduced to consider the corresponding outflow problem of the Navier-Stokes equations:
\begin{equation*}
\left\{
\begin{aligned}
   &\rho_t + (\rho u)_x = 0,   \\
   &\rho (u_t + uu_x ) + p_x = \mu u_{xx} ,   \\
   &\frac{R}{\gamma-1} \rho (\theta_t  + u \theta_x ) + p u_x = \mu u_x^2  +  \kappa \theta_{xx}  ,
\end{aligned}
\right.
\end{equation*}
with the initial-boundary values
\begin{equation*}
  (\rho, u, \theta)(x, 0) = (\rho_0, u_0, \theta_0)(x) \rightarrow (\rho_+, u_+, \theta_+), \quad \mathrm{as} \quad x \rightarrow +\infty,
\end{equation*}
\begin{equation*}
  (u, \theta)(0,t) = (u_-, \theta_-), \quad \forall \;  t \geq 0.
\end{equation*}
Hence, under the assumption $ E_+ = b_+ =  0 $, when time tends to infinity, it is reasonable for us to expect that the solutions to the outflow problem \eqref{yuanfangcheng}-\eqref{yuanfangcheng-chuzhi} and \eqref{yuanfangcheng-chuzhi-shuju}-\eqref{yuanfangcheng-bianzhi} asymptotically converge to the profiles the same as that of the Navier-Stokes equations.
Moreover, the cases for $ E_+ \neq 0 $ and $ b_+ \neq 0 $ which lead to more complex structures are left for study in future.

In the past three decades, there have been many works on the large-time behavior of solutions to the Cauchy problem of 1-D compressible Navier-Stokes equations (including its isentropic case) with some end constant states at far fields $ x = \pm \infty $ of initial data.
We refer interested readers to \cite{kawashima-matsumura-1985, Kawashima1986-feidengshang, liu-1986, liu-xin-1988, matsumura-nishihara-cmp-1992, nishihara-yang-zhao-2004, huang-matsumura-xin-2006, huang-xin-yang-2008, huang-matsumura-2009, duan-liu-zhao-2009, huang-li-matsumura-2010} and some references therein.
The above literatures show that the large-time behavior of solutions to the Cauchy problem with the far field constant states of initial data is basically governed by its corresponding Riemann solutions to the compressible Euler equations,
just as contact discontinuity and shock wave are replaced by the corresponding viscous contact wave and (shifted) viscous shock wave, respectively.

However, for the large-time behavior of solutions to an IBVP of the Navier-Stokes equations, there exists a different wave phenomenon from the Cauchy problem.
In fact, the authors of \cite{liu-matsumura-nishihara-1998, liu-nishihara-1997, liu-yu-1997} found a new wave phenomenon while studying the IBVP for scalar viscous conservation law.
This phenomenon appeared due to the boundary effect, and they named it boundary layer.
Since this wave's form is the stationary solution, other people also call it stationary solution.
From then on,
the investigation of the existence and stability of the boundary layer,
including the stability of its combinations with viscous hyperbolic waves has aroused many researchers' interests.
Later, Matsumura in \cite{matsumura.maa2001} gave the complete classification of the large-time behavior of the solutions for the compressible isentropic Navier-Stokes equations in terms of the far field states and the boundary data.
According to the sign of $u_-$, i.e. the value of fluid velocity at the boundary $ x=0 $,
the IBVP of the Navier-Stokes equations can be divided into three cases:
the outflow problem $(u_-<0)$, the inflow problem $(u_->0)$ and the impermeable wall problem $(u_-=0)$.
Since then, some conjectures in \cite{matsumura.maa2001} have been extensively investigated and verified for the isentropic and non-isentropic Navier-Stokes equations by many authors.
Here, we mention several works on the asymptotic stability analysis of wave patterns to the IBVP:
\cite{zhupc-cmp-2003, kawashima-zhu-jde-2008, huangfm-qin-jde-2009, kawashima-zhu-arma-2009, zhupc-m3-2010, qinxiaohong-2011, Nakamura-Yuge-inflow2011} for the outflow problem,
\cite{matsumura-inflow2001, huang-matsumura-shi-cmp2003, qin-wang-siam-2009, qin-wang-siam-2011, Nakamura-inflow2011} for the inflow problem
and \cite{matsumura-mei-arma-1999, matsumura-nishihara-qam-2000, huang-li-shi-cms-2010} for the impermeable wall problem.

These three kinds of IBVP are still important topics in the theory of fluid dynamics and plasma physics.
So it is meaningful and interesting to study the corresponding problems for the Navier-Stokes-Maxwell equations.
In the present paper, we only discuss the outflow problem.
The outflow boundary value $ u_-<0 $ means that fluid blows out from the boundary $x = 0$ with the velocity $ u_- $.
Thus this problem is called the outflow problem (see \cite{matsumura.maa2001}).
The outflow boundary condition implies that the characteristic of the hyperbolic equation $ \eqref{yuanfangcheng}_1 $ for the density $ \rho $ is negative around the boundary so that boundary
conditions on $ u $ and $ \theta $ to parabolic equations $ \eqref{yuanfangcheng}_2 $ and $ \eqref{yuanfangcheng}_3 $ are necessary and sufficient for the wellposedness of the hydrodynamic parts.
Motivated by \cite{qinxiaohong-2011, huangfm-qin-jde-2009, zhupc-m3-2010},
we will consider asymptotic stability of solutions towards the superposition of the boundary layer (including the nondegenerate case) and the 3-rarefaction wave under some smallness conditions and with the composite boundary condition of the electromagnetic fields.
To our knowledge, this can be viewed as the first result for the Navier-Stokes-Maxwell equations in this direction.

Here, we briefly give some remarks on our problem and review some key analytical techniques.
Compared with the result of \cite{qinxiaohong-2011} and \cite{zhupc-m3-2010} for compressible Navier-Stokes equations, the outflow problem for compressible non-isentropic Navier-Stokes-Maxwell equations is more complicated.

Due to the strong interaction between the fluid motion and the electromagnetic fields, the main difficulties to prove the nonlinear stability of wave patterns lie in the additional terms produced by the electrodynamic effects.
The first bad term about the electric field $ E $ and the magnetic field $ b $ we suffered is $-\int_0^t\int_{\mathbb{R}_+} (E+\psi b+ \hat{u} b)\psi b \,{\rm{d}}x{\rm{d}}\tau$.
But the lack of damping decay mechanism of the magnetic field $ b $ hinders us from obtaining the time-space integrable good term $ \int_0^t\int_{\mathbb{R}_+} b^2 \,\mathrm{d}x \mathrm{d}\tau $,
which is disadvantageous to the derivation of the zero-order energy estimates.
To overcome this obstacle, we try to use the structure of the Maxwell equations and package extra terms together with the bad term $-\int_0^t\int_{\mathbb{R}_+} (E+\psi b+ \hat{u} b)\psi b \,{\rm{d}}x{\rm{d}}\tau$ to produce a compound time-space integrable good term $\int_0^t\int_{\mathbb{R}_+} (E+\psi b+ \hat{u} b)^2 \,\mathrm{d}x{\rm{d}}\tau$,
which is crucial to obtain the zero-order energy estimates and is essential to get high-order energy estimates.
One can see Lemma \ref{lemmadijieenergy} of the zero-order energy estimates for details.

Secondly, we would encounter some obstacles under the composite boundary condition of the electromagnetic fields: $ (\sqrt{\varepsilon} E - b)(0,t) = 0 $.
For example, once using the Poincar$\acute{\rm e}$ type inequality \eqref{budengshi-Nikkuni} to estimate $ \int_0^t  \int_{\mathbb{R_+}} \tilde{u}_x^2  b^2  \,\mathrm{d}x\mathrm{d}\tau $,
then the bad term $ \delta ^3 \int_0^t b^2(0,\tau) \,\mathrm{d}\tau $ would arise.
But the absence of good term $ \int_0^t\int_{\mathbb{R}_+} b^2 \,\mathrm{d}x \mathrm{d}\tau $ makes it invalid to apply $ L^\infty_x $ Sobolev inequality \eqref{sobolevlwuqiong} to estimate $ b^2(0,t) $.
This requires a good term produces from the boundary estimates so that we can absorb the corresponding bad term by choosing $ \delta  $ suitably small.
In addition, to treat the boundary term $\int_0^t\int_{\mathbb{R_+}} -   \left(  \frac{1}{2} \hat{u}\varepsilon E^2 +  \frac{1}{2} \hat{u} b^2 + Eb   \right)_x  \,\mathrm{d}x \mathrm{d}\tau $ in \eqref{guji-boundary-eb1},
the sign of $ u_- $ is bad for this term. This fact also urges us to use the boundary condition $ (\sqrt{\varepsilon} E - b)(0,t) = 0 $ and the additional technical condition \eqref{jiedianxishu-bdyguji-1} to produce a boundary good term $ \frac{3}{4} \int_0^t  \sqrt{\varepsilon}E^2(0,\tau)  \,\mathrm{d}\tau $ in \eqref{guji-boundary-eb1}
such that we can employ it to absorb the former bad term $ \delta ^3 \int_0^t b^2(0,\tau) \,\mathrm{d}\tau $ indeed.
On the other hand, for boundary terms $ \int_0^t\int_{\mathbb{R_+}} - \left( E_x b_x  \right)_x  \,\mathrm{d}x \mathrm{d}\tau $ and $ \int_0^t\int_{\mathbb{R_+}}  2 \varepsilon \left( E E_x  \right)_x  \,\mathrm{d}x \mathrm{d}\tau  $,
we try to use the specific structure of the Maxwell equations and the composite boundary condition $ (\sqrt{\varepsilon} E - b)(0,t) = 0 $ to transform $ E_x(0,t) $ and $ b_x(0,t) $ into some suitable forms and get desired estimates eventually;
see \eqref{bdycons-Ebguanxi-2}-\eqref{guji-boundary-eb2} and \eqref{guji-boundary-eb3} for details.

Thirdly, for the boundary layer $\tilde{u}$,{}
we have to estimate the terms including the weight $ \tilde{u}_x $ such as $ - \int_0^t\int_{\mathbb{R}_+} \frac{1}{2} \tilde{u}_x \left( \varepsilon E^2 + b^2 \right) \,\mathrm{d}x\mathrm{d}\tau $ on the right-hand side of the inequality (see in \eqref{guji-jibennl-1}).
It is not workable to apply usual method to estimate this term directly.
Motivated by the idea of \cite{qinxiaohong-2011}: for each $ (\tilde{u}, \tilde{\theta})(x) $ there exists a constant $ M_0 \geq 1 $ just depending on $ u_-, \theta_-, \rho_*, u_*, \theta_* $ such that $ \tilde{u}_x \geq 0 $ and $ \tilde{\theta}_x \geq 0 $ on $ [M_0,+\infty) $,
we divide the integral into two parts: $  -  \left\{ \int_0^t \int_0^{M_0}     + \int_0^t\int_{M_0}^{+\infty}  \right\} \frac{1}{2} \tilde{u}_x \left(\varepsilon E^2 + b^2 \right) \,\mathrm{d}x\mathrm{d}\tau $.
This treatment avoids us to estimate the bad term $  \int_0^t\int_{M_0}^{+\infty} \frac{1}{2} \tilde{u}_x \left( \varepsilon E^2 + b^2 \right) \,\mathrm{d}x\mathrm{d}\tau $.
Then together with the boundary good term $ \frac{3}{4} \int_0^t  \sqrt{\varepsilon}E^2(0,\tau)  \,\mathrm{d}\tau $ in \eqref{guji-boundary-eb1},
we can obtain the desired estimates; see the proof of \eqref{fenqujianbanfa-1}-\eqref{yuxiang-diancichang-1} for details.

Fourthly, in order to absorb some nonlinear bad terms by some good terms concerning the electric field $ E $ or the magnetic field $ b $,
we require a technical condition \eqref{jiedianchangshuxiao} that $ \varepsilon $ is bounded for some specific positive constants $\bar C$.
Through some elaborate analysis, we can finally determine the value of the constant $\bar{C}$.
One can see the discussion about $ \varepsilon $ in \eqref{jiedianxishu-bdyguji-1}, \eqref{jiedianxishu-no.1}, \eqref{jiedianxishu-no.2}, \eqref{jiedianxishu-bdyguji-2}, \eqref{jiedianxishu-no.3} and \eqref{jiedianxishu-no.4} for details.
So far it is unclear how to remove such restriction for the nonlinear stability of wave patterns on the Navier-Stokes-Maxwell equations.


In the appendix, we provide a mathematical derivation for the one-dimensional model \eqref{yuanfangcheng}.
We also obtain additional four 1-D models, which remains to be studied in mathematics and physics in future.


\vspace{2mm}

\textbf{Notations:}
Throughout this paper, we denote positive constants generally large (respectively, generally small) independent of $ x $ and $ t $ by $ C $ (respectively, by $ c $).
And the character `C' and `c' may vary from line to line.
$\|\cdot\|_{L^p}$ stands the $L^p$-norm on the Lebesgue space ${L^p}(\mathbb{R_+})\;(1\leq p\leq \infty)$.
For the sake of convenience, we always denote $\|\cdot\|=\|\cdot\|_{L^2}$.
What's more, $H^k$ will be used to denote the usual Sobolev space $W^{k,2}(\mathbb{R_+})\;(k\in \mathbb{Z}_+)$ with respect to variable $x$.

\section{Preliminaries and main results}

\noindent Let
\begin{equation*}
  c(\rho,s) := \sqrt{p_\rho(\rho,s)} = \sqrt{R \gamma \theta} =: c(\rho,\theta),  \qquad   M(\rho,u,\theta) := \frac{\left| u \right| }{c},
\end{equation*}
which are called the local sound speed and the local Mach number. Let
\begin{equation*}
  c_+ :=  c(\rho_+, \theta_+) = \sqrt{R \gamma \theta_+},  \qquad   M_+ := M(\rho_+, u_+, \theta_+) = \frac{\left| u_+ \right| }{c_+},
\end{equation*}
which are called the sound speed and the Mach number at the far field $ x = +\infty $, respectively.
We divide the state space: the quarter 3D space $ \{(\rho, u, \theta) \;|\; \rho >0,\; \theta >0 \} $ into three parts:
\begin{equation*}
\left\{\begin{aligned}
& \Omega_{\mathrm{sub}} := \left\{ (\rho, u, \theta)  \left|\; \left| u \right|  < \sqrt{R \gamma \theta} \right.\right\} ,     \\[1mm]
& \Gamma_{\mathrm{tran}} := \left\{ (\rho, u, \theta) \left|\; \left| u \right|  = \sqrt{R \gamma \theta} \right.\right\} ,     \\[1mm]
& \Omega_{\mathrm{super}} := \left\{ (\rho, u, \theta) \left|\; \left| u \right|  > \sqrt{R \gamma \theta} \right.\right\} ,
\end{aligned}\right.
\end{equation*}
which are called subsonic, transonic and supersonic regions, respectively.
If we add the alternative condition $ u < 0 $ or $ u \geq 0 $ into these regions,
then we have six connected subsets $ \Omega_{\mathrm{sub}}^\pm $, $ \Gamma_{\mathrm{tran}}^\pm $ and $ \Omega_{\mathrm{super}}^\pm $.

\subsection{Boundary layer}\label{subsection2.1}

\noindent It is known that the corresponding hyperbolic system of (1.1) with (1.2) has three characteristic values
\begin{equation*}
  \lambda_1(\rho,u,\theta) = u - \sqrt{R \gamma \theta}, \qquad \lambda_2(\rho,u,\theta) = u, \qquad \lambda_3(\rho, u, \theta) = u + \sqrt{R \gamma \theta}.
\end{equation*}
For the outflow problem $u_- < 0$, one easily knows $\lambda_1 < \lambda_2 < 0$ at the boundary $x=0$.
So if $u_+ < 0$ and $u_-$ is sufficiently close to $u_+$ such that $u_- < 0$ also holds, then a stationary solution $ (\tilde{\rho}, \tilde{u}, \tilde{\theta}, \tilde{E}, \tilde{b})(x) $ to the outflow problem \eqref{yuanfangcheng}-\eqref{yuanfangcheng-chuzhi} and \eqref{yuanfangcheng-chuzhi-shuju}-\eqref{yuanfangcheng-bianzhi} is expected:
\begin{equation}\label{wentaijie-fangcheng-1}
\left\{ \begin{aligned}
& (\tilde{\rho} \tilde{u})_{x} =0,  \qquad x \in \mathbb{R}_{+},    \\[1mm]
& \tilde{\rho} \tilde{u} \tilde{u}_{x} +  \tilde{p}_{x} = \mu  \tilde{u}_{xx},   \\[1mm]
& \frac{R}{\gamma-1} \tilde{\rho} \tilde{u}  \tilde{\theta}_{x} + \tilde{p}  \tilde{u}_{x} = \mu \tilde{u}_{x}^{2}+\kappa  \tilde{\theta}_{xx},        \\[1mm]   
& \tilde{u}(0) = u_{-}, \quad \tilde{\theta}(0)=\theta_{-}, \quad(\tilde{\rho}, \tilde{u}, \tilde{\theta})(+\infty)=\left(\rho_{+}, u_{+}, \theta_{+}\right),     \\[1mm] 
& \inf _{x \in \mathbb{R}_{+}} \tilde{\rho}(x)>0, \quad \inf _{x \in \mathbb{R}_{+}} \tilde{\theta}(x)>0 ,
\end{aligned}\right.
\end{equation}    
with $\tilde{E} = \tilde{b} = 0 $, where $ \tilde{p} := p(\tilde{\rho},\tilde{\theta}) = R \tilde{\rho} \tilde{\theta} $.     
The stationary solution is usually called the boundary layer, see \cite{matsumura.maa2001}, \cite{qinxiaohong-2011} for example.
From the fact that $ \tilde{\rho}(x) >0 $ and $ u_-<0 $, we have
\begin{equation}\label{wentaijie-bianzhi}
  \rho_-: = \tilde{\rho}(0) = \frac{\rho_+u_+}{u_-}, \qquad \tilde{\rho}(x) = \frac{\rho_+ u_+}{\tilde{u}(x)}, \qquad \tilde{u}(x) < 0.
\end{equation}
Thus, \eqref{wentaijie-fangcheng-1} is equivalent to the coupling of \eqref{wentaijie-bianzhi} and the following ordinary differential equations:
\begin{equation}\label{wentaijie-fangcheng-2}
 \left\{\begin{aligned}
&  \tilde{u}_{x} =\frac{\rho_{+} u_{+}}{\mu}\left[\left(\tilde{u}-u_{+}\right)+R\left(\frac{\tilde{\theta}}{\tilde{u}}-\frac{\theta_{+}}{u_{+}}\right)\right], \qquad x \in \mathbb{R}_{+} ,      \\[2mm]
&  \tilde{\theta}_{x}  =  \frac{\rho_{+} u_{+}}{\kappa}   \left[  \frac{R \theta_{+}}{u_{+}}\left(\tilde{u}-u_{+}\right)       + \frac{R}{\gamma-1} \left(\tilde{\theta}-\theta_{+}\right)         - \frac{1}{2}\left(\tilde{u}-u_{+}\right)^{2} \right] ,       \\[2mm]
&  (\tilde{u}, \tilde{\theta})(0) = (u_{-},\theta_{-}),  \qquad(\tilde{u}, \tilde{\theta})(+\infty)=\left(u_{+}, \theta_{+}\right) ,
\end{aligned}\right.
\end{equation}
where $ p_+ := p(\rho_+,\theta_+) = R \rho_+ \theta_+ $.
The strength of the boundary layer is measured by
\begin{equation}
\delta  = \left| u_+ - u_- \right| + \left| \theta_+ - \theta_- \right| .
\end{equation}
Then we have the following lemma.


\begin{lemma}\label{wentaijie-lemma-1}$($Existence of the boundary layer$)$
Suppose that the boundary data $ (u_-, \theta_-) $ satisfy
\begin{equation}\label{tuxing-m+}
  \left(u_{-}, \theta_{-}\right) \in \mathcal{M}^{+}:=\left\{(u, \theta) \in \mathbb{R}^{2} ; \;  \left|\left(u-u_{+}, \theta-\theta_{+}\right)\right|<\delta_{0}\right\}
\end{equation}
for a certain positive constant $ \delta_0 $.
Notice that \eqref{tuxing-m+} is equivalent to the inequality $ \delta  < \delta_0  $.
\begin{itemize}
\item[$\mathrm{(i)}$]
For the supersonic case $M_+ > 1$, there exists a unique smooth solution $ (\tilde{u}, \tilde{\theta})(x) $ to the problem \eqref{wentaijie-fangcheng-2} satisfying
\begin{equation}\label{shuaijian-feituihua}
  \left| \left(\tilde{u}(x)-u_{+}, \tilde{\theta}(x)-\theta_{+}\right)^{(k)} \right|
  \leq C \delta  \mathrm{e}^{- c x},  \quad  ^{(k)}:= \frac{\mathrm{d}^{k}}{\mathrm{d}x^{k}},  \quad k=0,1,2, \cdots,
\end{equation}
where $c$ and $C$ are positive constants.


\item[$\mathrm{(ii)}$]
For the transonic case $M_+ = 1$,
there exists a centre-stable manifold $ \mathcal{M} \subset \mathcal{M}^+ $ consisting of two trajectories $\Gamma_{i} :=  \left(M_{i 1}, M_{i 2}\right)(\xi),\; i=1,2,\; \xi \in \mathbb{R}_{+}$, tangent to the line $ \mu u_+ (u-u_+) - (\gamma-1)\kappa (\theta - \theta_+) = 0 $ on the opposite directions at $ (u_+,\theta_+) $.
Depending on the location of $ (u_-,\theta_-) $, this case is divided into three subcases:

Subcase 1. For each $ (u_-,\theta_-)\in \mathcal{M}^+ $, $ (\tilde{u},\tilde{\theta}) \subset \mathcal{M} $, it holds that
\begin{equation}
  \left| \left(\tilde{u}(x)-u_{+}, \tilde{\theta}(x)-\theta_{+}\right)^{(k)} \right|
  \leq C \delta  \mathrm{e}^{- c x},    \quad k=0,1,2, \cdots.
\end{equation}

Subcase 2. For each $ (u_-,\theta_-) \in \mathcal{M}^+ $ satisfying $ \frac{\mu u_+ (u_+ - u_-)}{(\gamma-1) \kappa}  -  (\theta_+ - \theta_-)   < M_{i2}(\xi) $, where $ \xi $ is determined uniquely by $ M_{i 1}(\xi)=\frac{(\gamma-1)^{2} \kappa\left(u_{-}-u_{+}\right)}{R \mu \gamma+(\gamma-1)^{2} \kappa}+\frac{(\gamma-1) R \gamma \kappa\left(\theta_{-}-\theta_{+}\right)}{\left[R \mu \gamma+(\gamma-1)^{2} \kappa\right] u_{+}} $, $ i=1 $ or $ 2 $,
there exists a unique solution $ (\tilde{u},\tilde{\theta}) \subset \mathcal{M}^+ $ satisfying
\begin{equation}\label{shuaijian-tuihua}
  \left| \left(\tilde{u}(x)-u_{+}, \tilde{\theta}(x)-\theta_{+}\right)^{(k)} \right| \leq C \frac{\delta ^{k+1}}{(1+\delta  x)^{k+1}}, \quad k=0,1,2, \cdots,
\end{equation}
and
\begin{equation}
  \tilde{u}_x > 0, \quad \tilde{\theta}_x > 0  \quad \;\;\;  \mathrm{for}  \;\;  x \gg 1.
\end{equation}

Subcase 3. For each $ (u_-,\theta_-) \in \mathcal{M}^+ $, if it does not belong to Subcases 1 or 2,
then there exists no solution.

\item[$\mathrm{(iii)}$]
For the subsonic case $M_+ < 1$,
there exists a curve such that the unique smooth solution $ (\tilde{u}, \tilde{\theta})(x) $ to the problem \eqref{wentaijie-fangcheng-2} satisfying \eqref{shuaijian-feituihua}.

\end{itemize}

\end{lemma}

\begin{remark} For the transonic case $ M_+ = 1 $, $ \tilde{u}_x > 0 $ and $ \tilde{\theta}_x > 0 $ when $ x \gg 1 $ will be fundamental to obtain some energy estimates in the proof of Theorem \ref{wendingxing-bianjieceng} and \ref{wendingxing-fuhebo}. See Section \ref{proof-theorem2.3} for details.
The result of Case $\mathrm{(ii)}$ is borrowed from \cite{qinxiaohong-2011} and we skip the proof for brevity.
We should mention that Lemma \ref{wentaijie-lemma-1} was first obtained by Zhu et al. in \cite{zhupc-m3-2010} by using the central manifold theorem and Qin in \cite{qinxiaohong-2011} gave another proof of Case $\mathrm{(ii)}$ by employing the qualitative theory of ordinary differential equations.

\end{remark}


The asymptotic stability of the boundary layer $ (\tilde{\rho}, \tilde{u}, \tilde{\theta}, 0, 0) $ is stated in the following theorem.

\begin{theorem}\label{wendingxing-bianjieceng}
Assume that the boundary layer $ (\tilde{\rho}, \tilde{u}, \tilde{\theta}, 0, 0) $ exists under one of the following three conditions: $\mathrm{(i)}$ $ M_+ > 1 $; $\mathrm{(ii)}$ $ M_+ =1 $; $\mathrm{(iii)}$ $ M_+ < 1 $.
In addition, the dielectric constant $\varepsilon$ satisfies
\begin{equation}
0<\varepsilon<\bar{C}
\end{equation}
for some positive constant $\bar{C}$ $($depending only on $|u_-|$ and $|u_+|$$)$.
Then there exist two small positive constants $ \delta_1 $ and $ \varepsilon_1 $ which are independent of $T$, such that if $ 0 < \delta  < \min \{\delta_0, \delta_1 \}$ and the initial data satisfies
\begin{equation}
  \| (\rho_0, u_0, \theta_0, E_0, b_0) - (\tilde{\rho}, \tilde{u}, \tilde{\theta}, 0, 0) \|^2_{H^1}     \leq \varepsilon_1 ,
\end{equation}
then the initial-boundary value problem \eqref{yuanfangcheng}-\eqref{yuanfangcheng-chuzhi} and \eqref{yuanfangcheng-chuzhi-shuju}-\eqref{yuanfangcheng-bianzhi} has a unique global solution $ (\rho, u, \theta, E, b) $. 
Moreover, the solution $ (\rho, u, \theta, E, b) $ converges to the boundary layer $ (\tilde{\rho}, \tilde{u}, \tilde{\theta}, 0, 0) $ uniformly as time tends to infinity in the sense that:
\begin{equation}
\lim_{t \rightarrow +\infty}  \sup_{x\in \mathbb{R}_+} \left| (\rho, u, \theta, E, b)(x,t)  -  (\tilde{\rho}, \tilde{u}, \tilde{\theta}, 0, 0)(x) \right| = 0.
\end{equation}

\end{theorem}

\begin{remark}
Motivated by \cite{zhupc-m3-2010} and \cite{yin2019}, we expect to study the convergence rate of the solutions towards the \emph{non-degenerate} boundary layer for the 1-D Navier-Stokes-Maxwell equations.  
The main difficulty in the analysis lies that we can't get the weighted time-space integrable good terms $ \int_0^t\int_{\mathbb{R_+}} (1+\tau)^\xi W_{\hat{\nu},\beta} E^2 \,\mathrm{d}x \mathrm{d}\tau  $ and $ \int_0^t\int_{\mathbb{R_+}} (1+\tau)^\xi W_{\hat{\nu},\beta} b^2 \,\mathrm{d}x \mathrm{d}\tau  $ simultaneously, 
and only can obtain the dissipation good term $ \int_0^t\int_{\mathbb{R_+}} (1+\tau)^\xi W_{\hat{\nu},\beta} (E + ub)^2 \,\mathrm{d}x \mathrm{d}\tau $ instead, 
which is a typical feature of the Navier-Stokes-Maxwell equations with regularity-loss property. 
Precisely, obstacle occurs in the zero-order weighted energy estimates on $E $ and $ b $. 
By employing some similar argument as that of Lemma \ref{lemmadijieenergy}, then choosing $ \beta $, $ \varepsilon $ and $ \delta $ suitably small, it holds that 
\begin{align}\label{shijianquan} 
&\quad (1+t)^\xi \int_{\mathbb{R_+}}  W_{\hat{\nu},\beta} \left( \rho \eta(x,t) + \frac{1}{2}  \varepsilon E^2 + \frac{1}{2} b^2  +  \varepsilon E \tilde{u} b \right) \,\mathrm{d}x      \nonumber \\[1mm] 
&\quad  + c \int_0^t (1+\tau)^\xi \phi^2(0,\tau) \,\mathrm{d}\tau      + \left( 1- \left| u_- \right|\sqrt{\varepsilon} \right) \int_0^t (1+\tau)^\xi \sqrt{\varepsilon} E^2(0,\tau) \,\mathrm{d}\tau         \nonumber \\[1mm]  
&\quad  + c \beta \int_0^t\int_{\mathbb{R_+}} (1+\tau)^\xi W_{\hat{\nu}-1,\beta}  \left| (\phi,\psi,\zeta) \right|^2  \,\mathrm{d}x \mathrm{d}\tau       \nonumber \\[1mm]  
&\quad  + c \int_0^t\int_{\mathbb{R_+}} (1+\tau)^\xi W_{\hat{\nu},\beta} \left| (\psi_x, \zeta_x, E + u b) \right|^2  \,\mathrm{d}x \mathrm{d}\tau       \nonumber \\[1mm]  
& \leq  C \int_{\mathbb{R_+}} W_{\hat{\nu},\beta} \left| (\phi_0,\psi_0,\zeta_0,E_0,b_0) \right|^2 \,\mathrm{d}x      + C \delta \int_0^t (1+\tau)^\xi \, \varepsilon E^2(0,\tau) \,\mathrm{d}\tau     \nonumber \\[1mm] 
& + C \delta \int_0^t (1+\tau)^\xi \| (\phi_x,\sqrt{\varepsilon}E_x,b_x) \|^2  \,\mathrm{d}\tau         + C \int_0^t\int_{\mathbb{R_+}} \xi(1+\tau)^{\xi-1} W_{\hat{\nu},\beta} \left| (\phi, \psi, \zeta) \right|^2 \,\mathrm{d}x \mathrm{d}\tau     \nonumber \\[1mm]  
& -\beta \hat{\nu} \int_0^t\int_{\mathbb{R_+}} (1+\tau)^\xi W_{\hat{\nu}-1,\beta} \left( \frac{1}{2} \tilde{u} \varepsilon E^2 + \frac{1}{2} \tilde{u} b^2 + E b  \right)  \,\mathrm{d}x \mathrm{d}\tau      \nonumber \\[1mm]  
& + \int_0^t\int_{\mathbb{R_+}} \xi (1+\tau)^{\xi-1} W_{\hat{\nu},\beta} \left( \frac{1}{2} \varepsilon E^2  +  \frac{1}{2} b^2  +  \varepsilon E \tilde{u} b \right) \,\mathrm{d}x \mathrm{d}\tau.  
\end{align} 
Here, $ W_{\hat{\nu}, \beta} := (1+ \beta x)^{\hat{\nu}} $ is an algebraic weight and $ \eta(x,t) $ is defined in Lemma \ref{lemmadijieenergy}. 
In order to not only obtain the weighted space integrable good term $ (1+t)^\xi \int_{\mathbb{R_+}} W_{\hat{\nu},\beta} \left( \varepsilon E^2 +  b^2 \right) \,\mathrm{d}x $ from the first term on the left-hand side of \eqref{shijianquan}, 
but also apply an induction to $ \xi $ of the last term on the right-hand side of \eqref{shijianquan}, 
we hope that 
there exist two positive constants $ m $ and $ M $ such that 
\begin{equation}\label{tiaojian-danjifen}
 \left( \frac{1}{2} \varepsilon - a_1 \right) E^2 + \left( \frac{1}{2} - a_2 \right) b^2  +  \varepsilon E \tilde{u} b  \geq  m \left( \varepsilon E^2 + b^2 \right)
\end{equation}
and
\begin{equation}\label{tiaojian-dituixiang}
\left( \frac{1}{2} \varepsilon + k_1 \right) E^2 + \left( \frac{1}{2} + k_2 \right) b^2  +  \varepsilon E \tilde{u} b  \leq  M \left( E + u b \right)^2  
\end{equation}
for some given constants $ 0 \leq a_1 < \frac{1}{2} \varepsilon$, $ 0 \leq  a_2  < \frac{1}{2} $ and $ k_1,\,k_2 \geq 0 $. 
However, by means of positive semidefinite quadratic form, 
it is not hard to deduce $ 1 > \varepsilon \tilde{u}^2 $ from \eqref{tiaojian-danjifen} and $ \varepsilon \tilde{u}^2 > 1 $ from \eqref{tiaojian-dituixiang}, respectively. 
This implies \eqref{tiaojian-danjifen} and \eqref{tiaojian-dituixiang} can not be established simultaneously. 
Thus it seems unable to apply an induction to $ \xi $ of \eqref{shijianquan} to get the desired convergence rate just depending only on the weighted time-space integrable compound good term $ \int_0^t\int_{\mathbb{R_+}} (1+\tau)^\xi W_{\hat{\nu},\beta} \left( E + u b \right) ^2  \,\mathrm{d}x \mathrm{d}\tau $. 

\end{remark}


\subsection{Rarefaction wave}\label{subsection2.2}

\noindent It is well known that the 3-rarefaction wave curve through the right-hand side state $ (\rho_+, u_+, \theta_+) $ is
\begin{equation}
  R_{3}\left(\rho_{+}, u_{+}, \theta_{+}\right) := \left\{ (\rho, u, \theta)\; \left| \begin{aligned}
&\; 0<\rho<\rho_{+},\;\;  \rho^{1-\gamma} \theta=\rho_{+}^{1-\gamma} \theta_{+},         \\[2mm]
&\; u = u_{+}    + \int_{\rho_{+}}^{\rho} \sqrt{R \gamma \rho_{+}^{1-\gamma} \theta_{+}} \; \xi^{\frac{\gamma-3}{2}} \,\mathrm{d} \xi
\end{aligned}\right.\right\} .
\end{equation}
So for each pair of data $ (u_-, \theta_-) $ with the restriction condition
\begin{equation}
  u_{-} = u_{+}    + \int_{\rho_{+}}^{\left(\theta_{-} / \theta_{+}\right)^{\frac{1}{\gamma-1}} \rho_{+}} \sqrt{R \gamma \rho_{+}^{1-\gamma} \theta_{+}} \; \xi^{\frac{\gamma-3}{2}} \,\mathrm{d} \xi  ,
\end{equation}
there exists a unique $ \rho_- $ such that $ (\rho_-, u_-, \theta_-) \in R_3 (\rho_+,u_+,\theta_+)  $. The 3-rarefaction wave $ (\rho^r, u^r, \theta^r)(x/t) $ connecting $ (\rho_-, u_-, \theta_-) $ and $ (\rho_+, u_+, \theta_+) $ is the unique weak solution globally in time to the following Riemann problem:
\begin{equation}
\left\{\begin{aligned}
& \rho_{t}+(\rho u)_{x}=0,     \qquad x \in \mathbb{R},\;\; t>0,     \\[2mm]
& (\rho u)_{t}+\left(\rho u^{2}+p\right)_{x}=0,        \\[2mm]
& \left[\rho\left(\frac{R}{\gamma-1} \theta+\frac{1}{2} u^{2}\right)\right]_{t}+\left[\rho u\left(\frac{R}{\gamma-1} \theta+\frac{1}{2} u^{2}\right)+p u\right]_{x}=0,      \\[2mm]
& (\rho, u, \theta)(x, 0) =
              \left\{\begin{array}{ll}
              \left(\rho_{-}, u_{-}, \theta_{-}\right), \quad  & x<0,     \\[2mm]
              \left(\rho_{+}, u_{+}, \theta_{+}\right), \quad  & x>0,
              \end{array}\right.
\end{aligned}\right.
\end{equation}
with the electric field rarefaction wave $ \tilde{E}=0 $ and the magnetic field rarefaction wave $ \tilde{b}=0 $.
Here $ \theta_- < \theta_+  $ (or equivalently, $ \rho_- < \rho_+ $, $ u_- < u_+ $).
In addition, if $ \lambda_3(\rho_-,u_-,\theta_-) \geq 0 $, then the rarefaction wave is constant on $ (x,t) \in \mathbb{R}_- \times [0, + \infty) $.
For the outflow problem when $ u_- \in \Omega_{\mathrm{sub}}^- \cup \Gamma_{\mathrm{tran}}^- $, then $ \lambda_3(\rho_-,u_-,\theta_-) \geq 0 $.
Thus in this situation, one can expect that the solutions to the outflow problem converge towards 3-rarefaction wave which is similar to the Cauchy problem of \eqref{yuanfangcheng}.
To give the details of the large-time behavior of the solutions to the outflow problem, it is necessary to construct a smooth approximation $ (\bar{\rho},\bar{u},\bar{\theta})(x,t) $ of $ (\rho^r, u^r, \theta^r)(x/t) $. To this end, we borrow the idea from \cite{huangfm-qin-jde-2009} and \cite{matsumura-nishihara-cmp-1992}.

Consider the solution to the following Cauchy problem:
\begin{equation}\label{burgersfangcheng}
\left\{\begin{aligned}
& w_{t}+w w_{x}=0,     \qquad   x \in \mathbb{R},\;\; t>0,     \\[2mm]
& w(x, 0) = \left\{\begin{array}{ll}
            \displaystyle  w_{-},    & \;  x < 0,      \\[2mm]
            \displaystyle  w_{-}+C_{q} \delta_{r} \int_{0}^{\alpha x} y^{q} \mathrm{e}^{-y} \mathrm{~d} y,    & \;  x \geq 0.
            \end{array}\right.
\end{aligned}\right.
\end{equation}
Here $ \delta_r := w_+ - w_- > 0 $, $ 0 < \alpha < 1 \leq q  $ are two constants to be determined later, and $ C_q $ is a constant such that $ C_{q} \int_{0}^{+\infty} y^{q} \mathrm{e}^{-y} \,\mathrm{d}y = 1 $.
Let $ w_\pm = \lambda_3(\rho_\pm, u_\pm, \theta_\pm) $ and $ (\bar{\rho},\bar{u},\bar{\theta}, \bar{E}, \bar{b})(x,t) $ be defined as
\begin{equation}\label{gouzaoxishubo}
\left\{\begin{aligned}
&\left( \bar{u}+\sqrt{R \gamma \bar{\theta}} \right)  (x, t) = w(x, 1+t),    \qquad x  \in \mathbb{R},\;\;  t>0,      \\[2mm]
&\left(\bar{\rho}^{1-\gamma} \bar{\theta}\right)(x, t) = \rho_{+}^{1-\gamma} \theta_{+},      \\[2mm]
&\bar{u}(x, t)=u_{+}+\int_{\rho_{+}}^{\bar{\rho}(x, t)} \sqrt{R \gamma \rho_{+}^{1-\gamma} \theta_{+}} \; \xi^{\frac{\gamma-3}{2}} \,\mathrm{d} \xi,       \\[2mm]
& \bar{E}(x,t) = \bar{b}(x,t) = 0.
\end{aligned}\right.
\end{equation}
Due to $ w_- \geq 0 $ and \eqref{burgersfangcheng}, one has $ w(x,t) \equiv w_- $ on $ \mathbb{R}_- \times [0,+\infty)$.
From \eqref{gouzaoxishubo}, one easily knows that $ (\bar{\rho},\bar{u},\bar{\theta})(x,t) $ is constant on $ \mathbb{R}_-\times [0,+\infty) $ too.
Here we restrict $ (\bar{\rho},\bar{u},\bar{\theta})(x,t) $ in the half space $ \{x\geq0\} $ and still use $ (\bar{\rho},\bar{u},\bar{\theta})(x,t) $ to represent $ (\bar{\rho},\bar{u},\bar{\theta})(x,t)|_{x\geq 0} $.
Then one easily has
\begin{equation}
\left\{\begin{aligned}
& \bar{\rho}_{t}+(\bar{\rho} \bar{u})_{x} = 0 ,     \qquad x \in \mathbb{R}_{+}, \;\; t>0,      \\[2mm]
& (\bar{\rho} \bar{u})_{t}+\left(\bar{\rho} \bar{u}^{2} + \bar{p}\right)_{x} = 0,     \\[2mm]
& \left[\bar{\rho}\left(\frac{R}{\gamma-1} \bar{\theta}+\frac{1}{2} \bar{u}^{2}\right)\right]_{t}+\left[\bar{\rho} \bar{u}\left(\frac{R}{\gamma-1} \bar{\theta}+\frac{1}{2} \bar{u}^{2}\right)+\bar{p} \bar{u}\right]_{x} = 0,       \\[2mm]
& (\bar{\rho}, \bar{u}, \bar{\theta})(0, t) = \left(\rho_{-}, u_{-}, \theta_{-}\right), \quad(\bar{\rho}, \bar{u}, \bar{\theta})(x, 0) \rightarrow  \left\{\begin{array}{ll}
                          \left(\rho_{-}, u_{-}, \theta_{-}\right), &\;  x \rightarrow 0^{+},       \\[2mm]
                          \left(\rho_{+}, u_{+}, \theta_{+}\right), &\;  x \rightarrow +\infty,
                          \end{array}\right.
\end{aligned}\right.
\end{equation}
where $ \bar{p} := p(\bar{\rho},\bar{\theta}) = R \bar{\rho} \bar{\theta} $. Then the following lemma holds.

\begin{lemma}$($see \cite{huangfm-qin-jde-2009}\,$)$\label{xishubo-shuaijianlemma}
$ (\bar{\rho},\bar{u},\bar{\theta})(x,t) $ satisfies
\begin{itemize}
\item[$\mathrm{(i)}$]
$ 0 \leq \bar{\rho}_x,\, \bar{\theta}_x \leq C \bar{u}_x $,\;\; $ \left| (\bar{\rho}_{xx}, \bar{\theta}_{xx}) \right| \leq C (\left| \bar{u}_{xx}  \right| + \bar{u}_x^2  )  $, \; $ \forall \; (x,t)\in \mathbb{R}_+ \times \mathbb{R}_+ $;

\item[$\mathrm{(ii)}$]
For any $ p \; (1 \leq p \leq +\infty) $,
there exists a constant $ C_{pq} $ such that
\begin{align}
& \| (\bar{\rho}_x, \bar{u}_x, \bar{\theta}_x)(t) \|_{L^p}    \leq C_{pq} \min \{\delta_r \alpha^{1-\frac{1}{p}},  \delta_r^{\frac{1}{p}}(1+t)^{-1 + \frac{1}{p}} \},     \label{shuaijian-rare-yijiedao}   \\[2mm]
& \| (\bar{\rho}_{xx}, \bar{u}_{xx}, \bar{\theta}_{xx})(t) \|_{L^p}    \leq C_{pq} \min \{(\delta_r + \delta_r^2) \alpha^{2-\frac{1}{p}}, (\delta_r^{\frac{1}{p}} + \delta_r^{\frac{1}{q}} )(1+t)^{-1 + \frac{1}{q}} \};     \label{shuaijian-rare-erjiedao}
\end{align}

\item[$\mathrm{(iii)}$]
If $ x \leq (u_- + \sqrt{R \gamma \theta_-})(1+t)$, then $ (\bar{\rho},\bar{u},\bar{\theta})(x,t) - (\rho_-,u_-,\theta_-) \equiv 0 $;

\item[$\mathrm{(iv)}$]
$ \lim_{t\rightarrow +\infty} \sup_{x\in \mathbb{R}_+} \left| (\bar{\rho},\bar{u},\bar{\theta})(x,t) - (\rho^r,u^r,\theta^r)(\frac{x}{1+t}) \right| =0  $.

\end{itemize}

\end{lemma}

Now we can give the local stability of the 3-rarefaction wave:
\begin{theorem}\label{wendingxing-xishubo}
Assume $ (\rho_+, u_+, \theta_+) \in \Omega^-_{\mathrm{sub}} \cup \{u\geq 0\} $, $ \theta_- < \theta_+ $ and
\begin{equation*}
  u_- = u_+ + \int_{\rho_+}^{(\theta_-/\theta_+)^{\frac{1}{\gamma-1}  }\rho_+} \sqrt{R \gamma \rho_+^{1-\gamma} \theta_+} \;\xi^{\frac{\gamma-3}{2}} \,\mathrm{d}\xi \geq - \sqrt{R\gamma \theta_-}  \,.
\end{equation*}
In addition, the dielectric constant $\varepsilon$ satisfies
\begin{equation}
0<\varepsilon<\bar{C}
\end{equation}
for some positive constant $\bar{C}$ $($depending only on $|u_\pm|$ and $ \theta_\pm $$)$.
There is a suitably small constant $ \varepsilon_2 >0 $ which is independent of $T$, such that if
\begin{equation}
  \alpha + \| (\rho_0,u_0,\theta_0,E_0,b_0) - (\bar{\rho}_0,\bar{u}_0,\bar{\theta}_0, 0, 0) \|_{H^1} \leq \varepsilon_2,
\end{equation}
then the outflow problem \eqref{yuanfangcheng}-\eqref{yuanfangcheng-chuzhi} and \eqref{yuanfangcheng-chuzhi-shuju}-\eqref{yuanfangcheng-bianzhi} has a unique global solution $ (\rho, u, \theta, E, b)(x,t) $.
Furthermore,
\begin{equation}
  \lim_{t\rightarrow +\infty} \sup_{x\in \mathbb{R}_+} \left| (\rho,u,\theta,E,b)(x,t) - (\rho^r,u^r,\theta^r,0,0)(x/t) \right| = 0.
\end{equation}

\end{theorem}

\subsection{Superposition of the boundary layer and rarefaction wave}\label{subsection2.3}

\noindent Now let $ (\rho_+,u_+,\theta_+)\in \Omega^-_{\mathrm{sub}} \cup \{u \geq 0 \} $.
For $ (\rho_*,u_*,\theta_*) \in R_3(\rho_+,u_+,\theta_+) \cap (\Omega^-_{\mathrm{sub}} \cup \Gamma^-_{\mathrm{tran}}) $,
let $ S_* := \{ (\rho,u,\theta)\, | \, \rho u = \rho_* u_* \} $ be a family of surfaces.
From Section \ref{subsection2.2}, we know that
for each point $ (\rho_*,u_*,\theta_*) $ there exists a uniquely determined 3-rarefaction wave connecting it and $ (\rho_+,u_+,\theta_+) $.
Among the three variables, $ \rho_* $, $ u_* $ and $ \theta_* $, just one is independent, the other two can be determined accordingly. Precisely speaking, if let $ \rho_* $ be independent, then
\begin{equation}
 \rho_{*}<\rho_{+}, \qquad \rho_{*}^{1-\gamma} \theta_{*}=\rho_{+}^{1-\gamma} \theta_{+}, \qquad u_{*}=u_{+}+\int_{\rho_{+}}^{\rho_{*}} \sqrt{R \gamma \rho_{+}^{1-\gamma} \theta_{+}} \;\xi^{\frac{\gamma-3}{2}} \,\mathrm{d} \xi .
\end{equation}
Obviously, both $ u_* $ and $ \theta_* $ are strictly increasing and continuously differentiable with respect to $ \rho_* $.
From Section \ref{subsection2.1}, we can easily know that each boundary layer belongs to one surface of the family.
Consider the family $ S_* $ to be a function of $ \rho_* $, then
\begin{equation}
  \frac{\mathrm{d} S_{*}}{\mathrm{d} \rho_{*}}=u_{*}+\rho_{*} \frac{\mathrm{d} u_{*}}{\mathrm{d} \rho_{*}} = u_{*} + \sqrt{R \gamma \theta_{*}} .
\end{equation}
Due to $ (\rho_*,u_*,\theta_*) \in  \Omega^-_{\mathrm{sub}} \cup \Gamma^-_{\mathrm{tran}}  $, we know that $ R_3(\rho_+,u_+,\theta_+) $ and each one of $ S_* $ owns a unique intersection point, i.e., $ (\rho_*,u_*,\theta_*) $.
Moreover, all of $ S_* $ never intersect each other,
especially when $ (\rho_*,u_*,\theta_*) \in \Gamma^-_{\mathrm{tran}}  $.
For Case 1: if
\begin{equation}
  0 \neq \left|  u_- - u_+ - \int_{\rho_+}^{(\theta_-/\theta_+)^{\frac{1}{\gamma-1}  }\rho_+} \sqrt{R \gamma \rho_+^{1-\gamma} \theta_+} \;\xi^{\frac{\gamma-3}{2}} \,\mathrm{d}\xi   \right|  \ll 1 ,
\end{equation}
then it is expected that there exists a unique point $ (\rho_*,u_*,\theta_*) \in R_3(\rho_+,u_+,\theta_+)   \cap \Omega^-_{\mathrm{sub}}  $ such that
$\rho_*$, $u_*$, $\theta_*$, $ u_- $ and $ \theta_- $ satisfy \eqref{wentaijie-fangcheng-1} just when $(\rho_+,u_+,\theta_+)$ is replaced by $ (\rho_*,u_*,\theta_*) $ there.
For Case 2: when $ (\rho_*,u_*,\theta_*) \in R_3(\rho_+,u_+,\theta_+)   \cap \Gamma^-_{\mathrm{tran}} $, it holds that $ u_*  =  - \sqrt{R \gamma \theta_*} $, which means $ (\rho_*,u_*,\theta_*) $ is unique too.

Let
\begin{equation}
  (\hat{\rho}, \hat{u}, \hat{\theta})=(\tilde{\rho}, \tilde{u}, \tilde{\theta})+(\bar{\rho}, \bar{u}, \bar{\theta})-\left(\rho_{*}, u_{*}, \theta_{*}\right),
\end{equation}
with $ \hat{E}=\hat{b}=0 $, and the strength of boundary layer denoted by $ \delta  = \left| (u_*-u_-, \theta_* - \theta_-) \right| $.
Under the preliminaries above, we can state the third result.

\begin{theorem}\label{wendingxing-fuhebo}
Assume $ (\rho_+, u_+, \theta_+) \in \Omega^-_{\mathrm{sub}} \cup \{u\geq 0\} $,
$ (\rho_*,u_*,\theta_*) \in R_3(\rho_+,u_+,\theta_+) \cap (\Omega^-_{\mathrm{sub}} \cup \Gamma^-_{\mathrm{tran}}) $ and
$\rho_*$, $u_*$, $\theta_*$, $ u_- $, $ \theta_- $ satisfy \eqref{wentaijie-fangcheng-1} just when $ (\rho_+,u_+,\theta_+) $ there is replaced by $ (\rho_*,u_*,\theta_*) $.
In addition, the dielectric constant $\varepsilon$ satisfies
\begin{equation}\label{jiedianchangshuxiao}
0<\varepsilon<\bar{C}
\end{equation}
for some positive constant $\bar{C}$ $($depending only on $|u_\pm|$ and $ \theta_\pm $$)$.
There exist two small positive constants $ \delta_2 $ and $ \varepsilon_3 $ which are independent of $T$, such that if $ 0 < \delta  < \min \{\delta_0, \delta_2 \}$ and
\begin{equation}
  \alpha + \| (\rho_0,u_0,\theta_0,E_0,b_0) - (\hat{\rho}_0,\hat{u}_0,\hat{\theta}_0, 0, 0) \|_{H^1(\mathbb{R}_+)} \leq  \varepsilon_3  ,
\end{equation}
then the outflow problem \eqref{yuanfangcheng}-\eqref{yuanfangcheng-chuzhi} and \eqref{yuanfangcheng-chuzhi-shuju}-\eqref{yuanfangcheng-bianzhi} has a unique global solution $ (\rho, u, \theta, E, b)(x,t) $.
Furthermore,
\begin{equation}\label{dashijianxingwei-1}
  \lim_{t\rightarrow +\infty}\sup_{x\in \mathbb{R}_+} \left| (\rho,u,\theta)(x,t) - (\tilde{\rho},\tilde{u},\tilde{\theta})(x) - (\rho^r,u^r,\theta^r)({x}/{t})  +  (\rho_*,u_*,\theta_*) \right| = 0,
\end{equation}
and
\begin{equation}\label{dashijianxingwei-2}
  \lim_{t\rightarrow +\infty}\sup_{x\in \mathbb{R}_+} \left| (E,b)(x,t) - (0,0) \right| = 0.
\end{equation}

\end{theorem}


\begin{remark}
From \eqref{jiedianxishu-jieguo} for the case $ M_+ = 1 $, we can take the constant $ \bar{C} $ in \eqref{jiedianchangshuxiao} as
\begin{equation}\label{barc}
 \bar{C}  = \frac{1}{ 64 \max\{\left| u_- \right|, \left| u_+ \right|\}   \cdot \left( \max\{\left| u_- \right|, \left| u_+ \right|\}   +   \sqrt{ R \gamma \max\{\theta_-, \theta_+ \}} \right)  }   .
\end{equation}
Then for each given $\varepsilon$ satisfying the condition \eqref{jiedianchangshuxiao}, our system \eqref{yuanfangcheng} is explicitly well-defined.
On the one hand, when we take $ \max\{\left| u_- \right|, \left| u_+ \right|\}  $ suitably small,
 the dielectric
constant $ \varepsilon $ can be large enough, which can be seen from the conditions \eqref{jiedianchangshuxiao} and \eqref{barc} directly.
This fact can relax the requirement of smallness of $ \varepsilon $.
On the other hand, the conditions \eqref{jiedianchangshuxiao} and \eqref{barc} together can relax the restriction on $ \max\{\left| u_- \right|, \left| u_+ \right|\} $ as long as the dielectric constant $ \varepsilon $ is suitably small.
Thus, an interesting problem occurs,
that is
how to remove the technical condition \eqref{jiedianchangshuxiao} in future.
\end{remark}


\begin{remark}
For the compressible non-isentropic Navier-Stokes-Maxwell equations, the asymptotic stability of the wave patterns to the inflow problem and the impermeable wall problem can also be taken into account and remains to be studied in future.

\end{remark}



\section{Proofs of the theorems}\label{proof-theorem2.3}

\noindent It is easy to know that the proof of Theorems \ref{wendingxing-bianjieceng}, \ref{wendingxing-xishubo} is similar to and simpler than that of Theorem \ref{wendingxing-fuhebo} below, thus the details of Theorems \ref{wendingxing-bianjieceng}, \ref{wendingxing-xishubo} are omitted here.
And it is noted that Theorem \ref{wendingxing-fuhebo} concerns two cases of the boundary layer: one is non-degenerate, the other is degenerate.
If the boundary layer is not degenerate, i.e. decays exponentially, then employing a Poincar$\acute{\rm e}$-type inequality,
one easily knows that all the terms concerning the boundary layer are easier to be controlled than the counterparts of the degenerate case.
Hence, we only consider the proof of Theorem \ref{wendingxing-fuhebo} concerning the superposition of the degenerate boundary layer and the 3-rarefaction wave.

Recall
\begin{equation*}
  (\hat{\rho}, \hat{u}, \hat{\theta})(x,t)  =  (\tilde{\rho}, \tilde{u}, \tilde{\theta})(x)  +  (\bar{\rho}, \bar{u}, \bar{\theta})(x,t) - \left(\rho_{*}, u_{*}, \theta_{*}\right).
\end{equation*}
After some simple calculations, we can obtain
\begin{equation*}
\left\{\begin{aligned}
& \hat{\rho}_{t}+\hat{u} \hat{\rho}_{x}+\hat{\rho} \hat{u}_{x}=\hat{f},    \qquad  x \in \mathbb{R}_{+},\;\; t>0,     \\[2mm]
& \hat{\rho}\left(\hat{u}_{t} + \hat{u} \hat{u}_{x}\right) + \hat{p}_{x} = \mu \tilde{u}_{x x} + \hat{g},    \\[2mm]
& \frac{R}{\gamma-1} \hat{\rho}\left(\hat{\theta}_{t} + \hat{u} \hat{\theta}_{x}\right) + \hat{p} \hat{u}_{x} = \kappa \tilde{\theta}_{x x} + \mu \tilde{u}_{x}^{2} + \hat{h},    \\[2mm]
& (\hat{\rho}, \hat{u}, \hat{\theta})(0, t) = \left(\rho_{-}, u_{-}, \theta_{-}\right),  \qquad   (\hat{\rho}, \hat{u}, \hat{\theta})(+\infty, t) = \left(\rho_{+}, u_{+}, \theta_{+}\right),
\end{aligned}\right.
\end{equation*}
where $ \rho_- := (\rho_* u_*)/u_- $, and $ \rho_* $, $ u_* $, $ \theta_* $, $ u_- $, $ \theta_- $ satisfy \eqref{wentaijie-fangcheng-1}. Here $ \hat{p} := p(\hat{\rho},\hat{\theta}) = R \hat{\rho}\hat{\theta}$ and
\begin{equation*}
\left\{\begin{aligned}
& \hat{f} = \left(\bar{u}-u_{*}\right) \tilde{\rho}_{x}+\left(\bar{\rho}-\rho_{*}\right) \tilde{u}_{x}+\left(\tilde{u}-u_{*}\right) \bar{\rho}_{x}+\left(\tilde{\rho}-\rho_{*}\right) \bar{u}_{x},    \\[1mm]
& \hat{g} = \hat{\rho}\left[\left(\bar{u}-u_{*}\right) \tilde{u}_{x}+\left(\tilde{u}-u_{*}\right) \bar{u}_{x}\right]+\left(\bar{\rho}-\rho_{*}\right) \tilde{u} \tilde{u}_{x}+(\hat{p}-\tilde{p}-\bar{p})_{x}-\frac{\tilde{\rho}-\rho_{*}}{\bar{\rho}} \bar{p}_{x},     \\[1mm]
& \hat{h} = \frac{R}{\gamma-1} \hat{\rho}\left[\left(\bar{u}-u_{*}\right) \tilde{\theta}_{x}+\left(\tilde{u}-u_{*}\right) \bar{\theta}_{x}\right]+\frac{R}{\gamma-1}\left(\bar{\rho}-\rho_{*}\right) \tilde{u} \tilde{\theta}_{x}      \\[1mm]
& \qquad    + (\hat{p}-\tilde{p}) \tilde{u}_{x} + (\hat{p}-\bar{p}) \bar{u}_{x} - R \bar{\theta}\left(\tilde{\rho}-\rho_{*}\right) \bar{u}_{x} .
\end{aligned}\right.
\end{equation*}
Combining $ \bar{u}_x \geq 0 $ and \eqref{wentaijie-fangcheng-1}, we obtain that
\begin{equation}
|\hat{f}|  +  |\hat{g}|   +  |\hat{h}| \leq  C \left( (\bar{u} - u_*) \left| \tilde{u}_x \right|   +  \left| \tilde{u} - u_* \right| \bar{u}_x \right) .
\end{equation}
Define the perturbation as
\begin{equation}
 (\phi, \psi, \zeta, E, b)(x,t) = ( \rho-\hat{\rho}, u-\hat{u}, \theta - \hat{\theta}, E, b )(x,t).
\end{equation}
Then we transform the initial-boundary value problem \eqref{yuanfangcheng}-\eqref{yuanfangcheng-chuzhi} and \eqref{yuanfangcheng-chuzhi-shuju}-\eqref{yuanfangcheng-bianzhi} as
\begin{equation}\label{raodong-fuhebo}
 \left\{\begin{aligned}
& \phi_{t}+u \phi_{x}+\rho \psi_{x}=f,      \qquad  x \in \mathbb{R}_{+},\;\;  t>0,       \\[2mm]
& \rho\left(\psi_{t}+u \psi_{x}\right)+(p-\hat{p})_{x}=\mu \psi_{x x}  - (E + \psi b + \hat{u} b)b  + g  ,      \\[2mm]
& \frac{R}{\gamma-1} \rho\left(\zeta_{t}+u \zeta_{x}\right)+p \psi_{x} = \kappa \zeta_{x x}  +\mu \psi_{x}^{2}   + (E + \psi b + \hat{u} b)^2    + h,     \\[2mm]
& \varepsilon E_t - b_x + E + \psi b + \hat{u}b =0,    \\[2mm]
& b_t - E_x = 0,
 \end{aligned}\right.
\end{equation}
with the initial data
\begin{equation}\label{raodongchuzhi}
  (\phi_{0}, \psi_{0}, \zeta_{0}, E_0, b_0)(x) := (\phi, \psi, \zeta, E, b)(x, 0) \rightarrow(0,0,0,0,0), \quad  \text { as } x \rightarrow +\infty,
\end{equation}
and the boundary condition
\begin{equation}\label{raodongbdycon}
(\phi, \psi, \zeta, \sqrt{\varepsilon} E - b )(0, t)=(\rho(0,t) - \rho_-,0,0,0),
\end{equation}
where
\begin{equation}\label{raodongyuxiang-fuhebo}
\left\{\begin{aligned}
& f = -\hat{u}_{x} \phi-\hat{\rho}_{x} \psi  - \hat{f},       \\[1mm]
& g = -\rho \hat{u}_{x} \psi+\hat{p}_{x} \frac{\phi}{\hat{\rho}}-\mu \tilde{u}_{x x} \frac{\phi}{\hat{\rho}}+\mu \bar{u}_{x x}-\frac{\rho}{\hat{\rho}} \hat{g},     \\[1mm]
& h = -\frac{R}{\gamma-1} \rho \hat{\theta}_{x} \psi    - R \rho \hat{u}_{x} \zeta    -( \kappa \tilde{\theta}_{x x}+\mu \tilde{u}_{x}^{2} ) \frac{\phi}{\hat{\rho}}      \\[1mm]
& \quad\;\;\,   +\kappa \bar{\theta}_{x x}+2 \mu \hat{u}_{x} \psi_{x}+2 \mu \tilde{u}_{x} \bar{u}_{x}+\mu \bar{u}_{x}^{2}-\frac{\rho}{\hat{\rho}} \hat{h}.
\end{aligned}\right.
\end{equation}
For interval $I \subset [0, \infty)$, we define a function space $X(I)$ as
$$X(I) := \left \{
(\phi, \psi, \zeta, E, b) \; \left|\,
\begin{aligned}
& \;  \left(\phi, \psi, \zeta, E, b \right) \in L^{\infty}\left(I ; H^{1}(\mathbb{R}_+) \right),      \\[2mm]
& \;  \left( \phi_{x}, E_{x}, b_{x} \right)  \in L^{2}\left(I ; L^{2}(\mathbb{R}_+) \right),       \\[2mm]
& \;  \left(\psi_{x}, \zeta_{x}\right) \in L^{2}\left(I ; H^{1}(\mathbb{R}_+) \right)
\end{aligned}\right.
\right\}.$$

To prove Theorem \ref{wendingxing-fuhebo} for brevity, we only devote ourselves to deriving the uniform {\it a priori} estimates of the perturbation from the superposition of the degenerate boundary layer and the 3-rarefaction wave to the initial-boundary value problem \eqref{raodong-fuhebo}-\eqref{raodongyuxiang-fuhebo}.



\begin{proposition}\label{prop-1}$(${\it A priori} estimates$)$
Suppose that the boundary layer in Theorem $\ref{wendingxing-fuhebo}$ is degenerate.
Let $(\phi,\psi,\zeta,E,b)$ $\in X(0,T)$ be a smooth solution to the problem \eqref{raodong-fuhebo}-\eqref{raodongyuxiang-fuhebo} on $0\leq t\leq T$ for $T>0$.
There exist a positive constant $\bar{C}$ $($depending only on $|u_\pm|$ and $ \theta_\pm $$)$
and two suitably small positive constants $\delta_3$ and $\varepsilon_0$ such that
if the dielectric constant $\varepsilon$ and the strength of boundary layer $ \delta $ satisfy $ \varepsilon<\bar{C} $, $\delta < \min \{\delta_0, \delta_3 \}$ and the following {\it a priori} assumption holds:
\begin{equation}\label{xianyanjiashe}
 \sup_{0\leq t\leq T} \| (\phi, \psi, \zeta, E, b)(t) \|_{H^1}  \leq \varepsilon_0,
\end{equation}
then $ (\phi, \psi, \zeta, E, b)(x,t) $ satisfies
\begin{align}\label{zongnengliang-1}
  &\quad   \sup_{0\leq t\leq T} \left[ \| (\phi, \psi, \zeta, \sqrt{\varepsilon}E, b) \|^2_{H^1}      + \sqrt{\varepsilon} E^2(0,t)  \right]
     +  \int_0^T  \| (\phi_x, \psi_x, \zeta_x, E_x, b_x , \psi_{xx}, \zeta_{xx}) \|^2  \,\mathrm{d}\tau           \nonumber  \\[2mm]
  &\quad    +  \int_0^T  \left[  \phi^2(0,\tau)   + \phi^2_x(0,\tau)         + \sqrt{\varepsilon} E^2(0,\tau)   + \varepsilon^{\frac{3}{2}} E_\tau^2(0,\tau)      \right]   \,\mathrm{d}\tau         \nonumber  \\[2mm]
  &\quad    +  \int_0^T  \left[  \| E + \psi b + \hat{u} b \|^2       +  \| \sqrt{\bar{u}_x} (\phi, \psi, \zeta, \sqrt{\varepsilon}E, b )  \|^2    \right]   \,\mathrm{d}\tau         \nonumber  \\[2mm]
  &   \leq C \left( \| (\phi_0, \psi_0, \zeta_0, E_0, b_0) \|^2_{H^1}    +  \delta   +  \alpha^{\frac{1}{10}} \right) .
\end{align}

\end{proposition}


Owing to a smallness assumption \eqref{xianyanjiashe} on $ \| (\phi, \psi, \zeta, E, b)(t) \|_{H^1}  $, a quantity $ \| (\phi, \psi, \zeta, E, b)(t) \|_{L^{\infty}} $ is also sufficiently small, i.e.
\begin{equation}\label{raodongwuqiongmo}
\|(\phi, \psi, \zeta, E, b)(t) \|_{L^{\infty}}\leq \sqrt{2}\varepsilon_0,
\end{equation}
where we have used the following Sobolev inequality
\begin{equation}\label{sobolevlwuqiong}
\|f\|_{L^{\infty}}\leq \sqrt{2}\|f\|^\frac{1}{2}\|f_x\|^\frac{1}{2},\qquad {\rm for}\,\,f(x)\in H^1(\mathbb{R_+}).
\end{equation}


Once Proposition \ref{prop-1} is proved, we can close the {\it a priori} assumption \eqref{xianyanjiashe}.
The global existence of the solution to the initial-boundary value problem \eqref{raodong-fuhebo}-\eqref{raodongyuxiang-fuhebo} then follows from the
standard continuation argument based on the local existence and the {\it a priori} estimates.
For $0<\varepsilon<\bar{C}$, the estimate \eqref{zongnengliang-1} and the equations \eqref{raodong-fuhebo} imply that
\begin{equation*}
\int_0^\infty\left(\|(\phi_x,\psi_x,\zeta_x,E_x,b_x)(t)\|^2 + \left| \frac{\rm d }{{\rm d} t} \|(\phi_x,\psi_x,\zeta_x,E_x,b_x)(t)\|^2 \right| \right){\rm{d}}t <\infty,
\end{equation*}
which easily leads to
\begin{equation}\label{twuqiongfanshuling}
\lim_{t\to{+\infty}}\|(\phi_x,\psi_x,\zeta_x,E_x,b_x)(t)\|^2=0.
\end{equation}
Then using the Sobolev inequality \eqref{sobolevlwuqiong}, together with \eqref{twuqiongfanshuling}, directly implies the large time behavior of the solutions: \eqref{dashijianxingwei-1} and \eqref{dashijianxingwei-2}.

Firstly, due to \eqref{raodongwuqiongmo} and the smallness of $ \varepsilon_0 $, it is easy to deduce the following properties, which will be frequently used in the sequel.



\begin{lemma}\label{gezhongjie}
If the strength of boundary layer $\delta  = \left| (u_* - u_-,\;\theta_* - \theta_-) \right| $ is small enough, then
\begin{itemize}
\item[$\mathrm{(i)}$] $\tilde{u}(x)$, $ \bar{u}(x,t) $, $ \hat{u}(x,t) $ and $ u(x,t) $ satisfy
\begin{gather}
 \frac{3}{2} u_- < \tilde{u}(x) < \frac{1}{2} u_- < 0,\qquad    \frac{5}{4} u_- < \bar{u}(x,t) \leq u_+,       \nonumber  \\[2mm]
 \left| \hat{u}(x,t) \right| < \frac{3}{2} \max \{ \left| u_- \right|, \left| u_+ \right| \} ,  \qquad   \left| u(x,t) \right| < 2 \max \{ \left| u_- \right|, \left| u_+ \right| \}.   \nonumber
\end{gather}

\item[$\mathrm{(ii)}$] $\tilde{\theta}(x)$, $ \bar{\theta}(x,t) $, $ \hat{\theta}(x,t) $ and $ \theta(x,t) $ satisfy
\begin{gather}
 0 < \frac{1}{2} \theta_- < \tilde{\theta}(x) < \frac{3}{2} \theta_- , \qquad   0 < \frac{3}{4} \theta_- < \theta_* < \bar{\theta}(x,t) \leq \theta_+ ,     \nonumber \\[2mm]
 \frac{1}{2} \min\{\theta_-, \theta_+\} < \hat{\theta}(x,t)  < \frac{5}{4} \max\{\theta_-, \theta_+\},      \nonumber \\[2mm]
 \frac{1}{4} \min\{\theta_-, \theta_+\} < \theta(x,t)  < \frac{3}{2} \max\{\theta_-, \theta_+\}.    \nonumber
\end{gather}

\item[$\mathrm{(iii)}$] $ \rho_* $, $ \tilde{\rho}(x) $, $ \bar{\rho}(x,t) $, $ \hat{\rho}(x,t) $ and $ \rho(x,t) $ satisfy
\begin{gather}
0 <   \rho_+ \left( \frac{3}{4} \frac{\theta_-}{\theta_+} \right)^{\frac{1}{\gamma-1}}    < \rho_* \leq \rho_+,        \nonumber  \\[2mm]
\frac{1}{2} \rho_+ \left( \frac{3}{4} \frac{\theta_-}{\theta_+} \right)^{\frac{1}{\gamma-1}} < \tilde{\rho}(x) < \frac{3}{2} \rho_+,  \qquad     \rho_+ \left( \frac{3}{4} \frac{\theta_-}{\theta_+} \right)^{\frac{1}{\gamma-1}}  < \bar{\rho}(x,t) \leq \rho_+,        \nonumber \\[2mm]
\frac{1}{2} \rho_+  \left( \frac{3}{4} \frac{\theta_-}{\theta_+} \right)^{\frac{1}{\gamma-1}}   < \hat{\rho}(x,t)  < \frac{3}{2} \rho_+,    \qquad    \frac{1}{4} \rho_+ \left( \frac{3}{4} \frac{\theta_-}{\theta_+} \right)^{\frac{1}{\gamma-1}}    <   \rho(x,t)   < \frac{7}{4} \rho_+.       \nonumber
\end{gather}

\end{itemize}

\end{lemma}

Next, the following useful lemma plays an important role in the proof of the {\it a priori} estimates.
\begin{lemma}\label{lemma-daiquanL2}

Assume functions $ z(x, t) \in H_x^1(\mathbb{R}_+) $, then
\begin{itemize}
\item[$\mathrm{(i)}$] for the boundary layer $ \tilde{u}(x) $ satisfying Subcase 2 of the transonic case $M_+ = 1$, it holds
\begin{equation}\label{guji-bianjieceng^2quan}
  \int_{\mathbb{R_+}} \tilde{u}_x^2  z^2  \,\mathrm{d}x   \leq C \delta ^3  z^2(0,t)   + C \delta ^2 \| z_x \|^2.
\end{equation}
\item[$\mathrm{(ii)}$] for the smooth approximate rarefaction wave $ \bar{u}(x,t) $, it holds
\begin{equation}\label{guji-xishubo^2quan}
  \int_{\mathbb{R_+}} \bar{u}_x^2 z^2 \,\mathrm{d}x  \leq   C  \alpha^{\frac{1}{3}} \| z \|  \left(  \| z_x \|^2   +  (1+t)^{-\frac{4}{3}}  \right) .
\end{equation}
\end{itemize}

\end{lemma}

\begin{proof}[Proof]
\eqref{guji-bianjieceng^2quan} follows easily by employing the following Poincar$\acute{\rm e}$ type inequality used by Nikkuni \cite{Nikkuni1999}:
\begin{equation}\label{budengshi-Nikkuni}
  \left| z(x,t) \right| \leq \left| z(0,t) \right| + x^{\frac{1}{2}} \| z_x \| ,  \quad  z(x, t) \in H_x^1(\mathbb{R}_+)
\end{equation}
and the decay rate of the degenerate boundary layer.
For \eqref{guji-xishubo^2quan},
by using the Sobolev inequality \eqref{sobolevlwuqiong} and \eqref{shuaijian-rare-yijiedao} in Lemma \ref{xishubo-shuaijianlemma}, we have
\begin{align}
\int_{\mathbb{R_+}} \bar{u}_x^2 z^2 \,\mathrm{d}x
& \leq  \| z \|^2_{L^{\infty}_x} \| \bar{u}_x \|^2    \leq  2 \| z \| \cdot \| z_x \|\cdot \| \bar{u}_x \|^2     \leq  C \| z \|\cdot \| z_x \| \cdot \alpha^{\frac{1}{3}} (1+t)^{-\frac{2}{3}}        \nonumber     \\[2mm]
&   \leq   C \alpha^{\frac{1}{3}}  \| z \|  \left(  \| z_x \|^2   +  (1+t)^{-\frac{4}{3}}  \right) .    \nonumber
\end{align}

\end{proof}

\subsection{Zero-order energy estimates}

\begin{lemma} \label{lemmadijieenergy}
Suppose that the conditions in Proposition \ref{prop-1} hold.
Then for all $ 0 < t < T $, we have the following zero-order energy estimates
\begin{align}\label{jibennengliang}
& \| (\phi, \psi, \zeta, \sqrt{\varepsilon}E, b ) \|^2           + \int_0^t      \left[ \phi^2(0, \tau)     + \sqrt{\varepsilon} E^2(0,\tau)      + \left\| \left(   \psi_x, \zeta_x,  E + \psi b + \hat{u} b   \right)  \right\|^2     \right]       \,\mathrm{d}\tau       \nonumber  \\[2mm]
&\quad  +  \int_0^t  \| \sqrt{\bar{u}_x}( \phi, \psi, \zeta, \sqrt{\varepsilon}E, b) \|^2  \,\mathrm{d}\tau       \leq  C \| (\phi_0, \psi_0, \zeta_0, E_0, b_0) \|^2              + C (\delta  + \alpha^{\frac{1}{10}})        \nonumber  \\[2mm]
&\qquad\qquad\qquad\qquad\qquad\qquad\qquad\qquad   + C (\delta  + \alpha^{\frac{1}{10}}) \int_0^t \| (\phi_x, \sqrt{\varepsilon} E_x, b_x  ) \|^2  \,\mathrm{d}\tau  .
\end{align}


\end{lemma}

\begin{proof}[Proof]
The proof of the zero-order energy estimates in Lemma \ref{lemmadijieenergy} includes the following two steps.

{\bf Step 1:}
Inspired by the work of the compressible Navier-Stokes equations in \cite{Kawashima1986-feidengshang} and \cite{qinxiaohong-2011}, we take $ \Phi(s) = s - 1 -\ln s $\, for $s>0$, and set
$$
\eta(x,t) =  \frac{1}{2} \psi^2  + R \hat{\theta} \Phi \left( \frac{\hat{\rho}}{\rho} \right)  + \frac{R}{\gamma -1 } \hat{\theta} \Phi \left( \frac{\theta}{\hat{\theta}} \right).
$$
Tedious calculations give rise to
\begin{equation}\label{shangliudui-0}
( \rho\cdot\eta(x,t) )_t      +  H_x     +  \mu \frac{\hat{\theta}}{\theta} \psi_x^2    + \kappa \frac{\hat{\theta}}{\theta^2} \zeta_x^2     =        Q  + (E + \psi b + \hat{u}b)^2 \frac{\zeta}{\theta} - (E + \psi b + \hat{u}b)\psi b ,
\end{equation}
where
\begin{align}
H & =  \rho u \cdot \eta(x,t) +  ( p - \hat{p} ) \psi  - \mu \psi \psi_x  - \kappa \frac{\zeta \zeta_x}{\theta},        \nonumber  \\[2mm]
Q & =   -\frac{R \hat{\theta} \hat{\rho}_{x} \phi \psi}{\hat{\rho}}       - \frac{R \hat{\theta}  \hat{f} \phi}{\hat{\rho}}        +  g \psi          + \left( 2 \mu \hat{u}_{x} \psi_{x}     +     2 \mu \tilde{u}_{x} \bar{u}_{x}    + \mu \bar{u}_{x}^{2}       + \kappa \bar{\theta}_{x x}\right) \frac{\zeta}{\theta}       \nonumber  \\[2mm]
&\quad     -\left(\frac{p \hat{\theta}_{x} \psi}{(\gamma-1) \hat{\theta}} + \frac{\kappa \tilde{\theta}_{x x}(p-\hat{p})+\mu \tilde{u}_{x}^{2}(p-\hat{p})   + p \hat{h}}{\hat{p}}\right) \frac{\zeta}{\theta}           +\kappa  \frac{\zeta \zeta_{x} \hat{\theta}_{x}}{\theta^2}      \nonumber    \\[2mm]
&\quad   + \rho\left(\frac{R}{\gamma-1} \hat{\theta}_{x} \psi+\frac{\kappa \tilde{\theta}_{x x}+\mu \tilde{u}_{x}^{2}-\hat{p} \hat{u}_{x}   +  \hat{h} }{\hat{\rho}}\right)\left[(\gamma-1) \Phi\left(\frac{\hat{\rho}}{\rho}\right)+\Phi\left(\frac{\theta}{\hat{\theta}}\right)\right]  .   \nonumber 
\end{align}



Next, we use the specific structure of the Maxwell equations to treat the terms about electromagnetic fields on the right-hand side of \eqref{shangliudui-0}.
Multiplying $\eqref{raodong-fuhebo}_{4}$ by $E$ and
$\eqref{raodong-fuhebo}_{5}$ by $b$ respectively, then summing them up gives
\begin{equation}\label{coueub-1}
  \frac{1}{2} (\varepsilon E^2 + b^2)_t  - (Eb)_x  + (E + \psi b + \hat{u}b)E = 0.
\end{equation}
Multiplying $\eqref{raodong-fuhebo}_{4}$ by $ \hat{u} b$ and applying $ \eqref{raodong-fuhebo}_{5} $ deduces
\begin{align}\label{coueub-2}
  ( \varepsilon E \hat{u} b)_t   - \frac{1}{2} \left( \hat{u} (\varepsilon E^2 + b^2) \right)_x        + \frac{1}{2} \bar{u}_x (\varepsilon E^2 + b^2 )    +  (E + \psi b + \hat{u}b) \hat{u} b       \nonumber  \\[2mm]
   =    - \frac{1}{2}   \tilde{u}_x \left( \varepsilon E^2 + b^2  \right)       +  \varepsilon E b \hat{u}_t .
\end{align}
Then summing \eqref{coueub-1}, \eqref{coueub-2} and \eqref{shangliudui-0} up gives
\begin{align}\label{wode-shangliudui}
& \left( \rho\cdot \eta(x,t)    +  \frac{1}{2} \varepsilon E^2    + \frac{1}{2} b^2    + \varepsilon E \hat{u} b \right) _t           + \left(  H  -   \left(  \frac{1}{2} \hat{u}\varepsilon E^2 +  \frac{1}{2} \hat{u} b^2 + Eb   \right) \right)_x          +  \mu \frac{\hat{\theta}}{\theta} \psi_x^2        \nonumber  \\[2mm]
&  + \kappa \frac{\hat{\theta}}{\theta^2} \zeta_x^2       + \frac{1}{2} \bar{u}_x (\varepsilon E^2 + b^2 )      +  \left( E + \psi b + \hat{u} b  \right)^2     =  Q     - \frac{1}{2} \tilde{u}_x (\varepsilon E^2 + b^2 )      + \varepsilon E b \bar{u}_t        \nonumber  \\[2mm]
& \qquad\qquad\qquad\qquad\qquad\qquad\qquad\qquad\qquad\qquad\quad  +  \left( E + \psi b + \hat{u} b  \right)^2 \frac{\zeta}{\theta} .
\end{align}
Set $ \beta_1 := \max\{\left| u_- \right|, \left| u_+ \right|\} $.
The boundary condition $\sqrt{\varepsilon}E(0,t)  =  b(0,t)  $ implies the following boundary estimate
\begin{align}\label{guji-boundary-eb1}
  &\quad \int_0^t\int_{\mathbb{R_+}} -   \left(  \frac{1}{2} \hat{u}\varepsilon E^2 +  \frac{1}{2} \hat{u} b^2 + Eb   \right)_x  \,\mathrm{d}x \mathrm{d}\tau       \nonumber \\[2mm]
  &=  \int_0^t \left( \frac{1}{2} u_- \varepsilon E^2(0,\tau)    +  \frac{1}{2} u_- b^2(0,\tau)    + E(0,\tau)b(0,\tau) \right) \,\mathrm{d}\tau         \nonumber \\[2mm]
  &=  \int_0^t \left(  u_- \varepsilon E^2(0,\tau)       + \sqrt{\varepsilon} E^2(0,\tau)  \right) \,\mathrm{d}\tau        = \int_0^t      \left( 1 - \left| u_- \right| \sqrt{\varepsilon}  \right) \sqrt{\varepsilon} E^2(0,\tau)   \,\mathrm{d}\tau   \nonumber \\[2mm]
  & \geq  \int_0^t      \left( 1 -  \beta_1 \sqrt{\varepsilon}  \right) \sqrt{\varepsilon} E^2(0,\tau)   \,\mathrm{d}\tau    \geq   \frac{3}{4} \int_0^t  \sqrt{\varepsilon}E^2(0,\tau)  \,\mathrm{d}\tau ,
\end{align}
where in the last inequality we have taken
\begin{equation}\label{jiedianxishu-bdyguji-1}
 1 -  \beta_1 \sqrt{\varepsilon} \geq  \frac{3}{4},  \quad   \mathrm{i.e.}   \quad  \beta_1^2 \varepsilon \leq \frac{1}{16} .
\end{equation}
Moreover, using the boundary condition \eqref{raodongbdycon},
thanks to the good sign of $ u_- $, we can get
\begin{align}\label{guji-boundary-ns1}
 \int_0^t \int_{\mathbb{R_+}}  H_x   \,\mathrm{d}x \mathrm{d}\tau    =  -   \int_0^t  H(0,\tau)  \,\mathrm{d}\tau
  & =  -\, u_-  R \rho_- \theta_-  \int_0^t \Phi\left(\frac{\hat{\rho}}{\rho}\right)(0,\tau) \,\mathrm{d}\tau      \nonumber  \\[2mm]
  & \geq    c \int_0^t \phi^2(0,\tau) \,\mathrm{d}\tau .
\end{align}
Integrating \eqref{wode-shangliudui} with respect to $ x $ and $ t $, then plugging \eqref{guji-boundary-eb1} and \eqref{guji-boundary-ns1} into the resulting equality, and employing \eqref{raodongwuqiongmo}, we obtain
\begin{align}\label{guji-jibennl-1}
&\quad   \int_{\mathbb{R_+}} \left( \rho\cdot \eta(x,t)    +  \frac{1}{2} \varepsilon E^2    + \frac{1}{2} b^2    \right) \,\mathrm{d}x         +     \frac{1}{2} \int_0^t \| \sqrt{\bar{u}_x}(\sqrt{\varepsilon}E, b) \|^2  \,\mathrm{d}\tau                          \nonumber  \\[2mm]
&\quad    + \frac{3}{4} \int_0^t  \sqrt{\varepsilon}E^2(0,\tau)  \,\mathrm{d}\tau       +  c_1 \int_0^t      \left[ \phi^2(0, \tau)       + \left\| \left(   \psi_x, \zeta_x,  E + \psi b + \hat{u} b   \right)  \right\|^2  \right]       \,\mathrm{d}\tau           \nonumber  \\[2mm]
& \leq  C \| (\phi_0, \psi_0, \zeta_0, E_0, b_0) \|^2       +  \int_0^t\int_{\mathbb{R_+}} Q \,\mathrm{d}x \mathrm{d}\tau        +  \int_0^t\int_{\mathbb{R_+}}  \varepsilon E b \bar{u}_t   \,\mathrm{d}x \mathrm{d}\tau             \nonumber  \\[2mm]
&\quad      - \int_0^t\int_{\mathbb{R_+}}    \frac{1}{2} \tilde{u}_x (\varepsilon E^2 + b^2 )  \,\mathrm{d}x \mathrm{d}\tau         + \int_{\mathbb{R_+}}  \left| \varepsilon E \hat{u} b \right| \,\mathrm{d}x   .
\end{align}


{\bf Step 2:} First, we will estimate the terms $ -\int_0^t\int_{\mathbb{R_+}}  \frac{1}{2} \tilde{u}_x (\varepsilon E^2 + b^2 ) \,\mathrm{d}x \mathrm{d}\tau  $ and $  \int_0^t\int_{\mathbb{R_+}} \varepsilon E b \bar{u}_t  \,\mathrm{d}x \mathrm{d}\tau   $ on the right-hand side of \eqref{guji-jibennl-1}.
For $ - \int_0^t\int_{\mathbb{R_+}}  \frac{1}{2} \tilde{u}_x (\varepsilon E^2 + b^2 ) \,\mathrm{d}x \mathrm{d}\tau  $,
motivated by the work of \cite{qinxiaohong-2011},
we note that in Subcase 2 of Case (ii) $ (\tilde{u},\tilde{\theta})(x) $ is parallel to the vector $ (-\mu u_*, (\gamma-1)\kappa ) $ at $ (u_*, \theta_*) $.
Hence, for each $ (\tilde{u}, \tilde{\theta})(x) $ there exists a constant $ M_0 \geq 1 $ just depending on $ u_-, \theta_-, \rho_*, u_*, \theta_* $ such that if $ x> M_0 $, then $ (\tilde{u}, \tilde{\theta})(x) \nearrow (u_*, \theta_*) $ as $ x \rightarrow +\infty $.
This implies $ \tilde{u}_x \geq 0 $ and $ \tilde{\theta}_x \geq 0 $ on $ [M_0,+\infty) $.
Thanks to this important observation, we divide the integral into two parts:
\begin{align}\label{fenqujianbanfa-1}
- \int_0^t\int_{\mathbb{R_+}}    \frac{1}{2} \tilde{u}_x (\varepsilon E^2 + b^2 )    \,\mathrm{d}x \mathrm{d}\tau  =    -  \left\{      \int_0^t \int_0^{M_0}     +   \int_0^t\int_{M_0}^{+\infty}      \right\}     \frac{1}{2} \tilde{u}_x (\varepsilon E^2 + b^2 )      \,\mathrm{d}x\mathrm{d}\tau        =: I_1 + I_2.
\end{align}
For $ I_1 $, employing the Poincar$\acute{\rm e}$ type inequality \eqref{budengshi-Nikkuni} on $ E $ and $ b $,
then using the boundary condition $\sqrt{\varepsilon}E(0,t) = b(0,t)$,
we can derive
\begin{align}\label{guji-I1}
  I_1
  & \leq C   \int_0^t\int_0^{M_0}   \frac{\delta ^2}{(1+\delta x  )^2 } (\varepsilon E^2 + b^2 )     \,\mathrm{d}x \mathrm{d}\tau            \nonumber     \\[2mm]
  & \leq  C \int_0^t \left( \varepsilon E^2(0,\tau) + b^2(0,\tau) \right)   \left( \int_0^{M_0} \frac{\delta ^2 }{(1+\delta x  )^2 }  \,\mathrm{d}x  \right)  \,\mathrm{d}\tau      \nonumber     \\[2mm]
  & \quad    +  C   \int_0^t   \left( \varepsilon \| E_x \|^2  +  \| b_x \|^2 \right)   \left( \int_0^{M_0} \frac{\delta ^2 x}{(1+\delta x  )^2 }  \,\mathrm{d}x  \right)  \mathrm{d}\tau          \nonumber     \\[2mm]
  &  \leq    C \delta  \int_0^t \varepsilon E^2(0,\tau) \,\mathrm{d}\tau      +    C  \delta    \int_0^t   \left( \varepsilon \| E_x \|^2  +  \| b_x \|^2 \right)   \mathrm{d}\tau ,
\end{align}
where in the last inequality we have used the following simple inequality:
\begin{equation*}
   0 \leq  \ln (1+x) \leq x, \quad x \in \mathbb{R}_+.
 \end{equation*}
Moreover, $ \tilde{u}_x \geq 0 \; (x > M_0) $ implies
\begin{equation}\label{guji-I2}
  I_2 =  - \int_0^t \int_{M_0}^{+\infty} \frac{1}{2} \tilde{u}_x (\varepsilon E^2 + b^2 )  \,\mathrm{d}x \mathrm{d}\tau  \leq  0.
\end{equation}
Combining \eqref{guji-I1} and \eqref{guji-I2}, we can get
\begin{align}\label{yuxiang-diancichang-1}
 - \int_0^t \int_{\mathbb{R_+}} \frac{1}{2} \tilde{u}_x (\varepsilon E^2 + b^2 )  \,\mathrm{d}x\mathrm{d}\tau    \leq  C \delta  \int_0^t \varepsilon E^2(0,\tau) \,\mathrm{d}\tau           + C \delta  \int_0^t \| (\sqrt{\varepsilon} E_x, b_x) \|^2 \,\mathrm{d}\tau  .
\end{align}


For $  \int_0^t\int_{\mathbb{R_+}} \varepsilon E b \bar{u}_t  \,\mathrm{d}x \mathrm{d}\tau   $, we have
\begin{align}\label{guji-diancichang-xishubo}
 & \int_0^t \int_{\mathbb{R_+}}  \varepsilon E b \bar{u}_t   \,\mathrm{d}x \mathrm{d}\tau     =    \int_0^t \int_{\mathbb{R_+}} \varepsilon (E + ub - ub)b \bar{u}_t \,\mathrm{d}x \mathrm{d}\tau         \nonumber  \\[2mm]
 &=  - \int_0^t \int_{\mathbb{R_+}} \varepsilon u b^2 \bar{u}_t \,\mathrm{d}x \mathrm{d}\tau          +  \int_0^t \int_{\mathbb{R_+}} \varepsilon (E + u b) b \bar{u}_t \,\mathrm{d}x \mathrm{d}\tau      =: J_1  +  J_2 .
\end{align}
Set $ \beta_2 := \max\{\theta_-, \theta_+ \} $, $ \beta_3 :=  \beta_1 + \sqrt{R \gamma \beta_2} $ and recall $ \beta_1 = \max\{\left| u_- \right|, \left| u_+ \right|\} $.
Due to (i) of Lemma \ref{gezhongjie}, elementary calculations give rise to
\begin{align}\label{guji-baru_t}
\left|  \bar{u}_t \right|
& =  \left| - \bar{u} \bar{u}_x  -  \frac{\bar{p}_x}{\bar{\rho}} \right|    =  \left| - \left( \bar{u} + \sqrt{ R \gamma \bar{\theta}}  \right) \bar{u}_x \right|     \leq \left( \frac{5}{4} \beta_1 + \sqrt{R \gamma \beta_2}  \right) \bar{u}_x          \nonumber \\[1mm]
&\qquad\qquad\qquad\qquad\qquad\qquad    \leq   \frac{5}{4} \left( \beta_1 + \sqrt{R \gamma \beta_2}  \right)  \bar{u}_x    =  \frac{5}{4} \beta_3 \bar{u}_x .
\end{align}
Then using (i) of Lemma \ref{gezhongjie} and \eqref{guji-baru_t} gives that
\begin{equation}\label{guji-J1}
  \left| J_1 \right|    \leq     \frac{5}{2} \beta_1 \beta_3 \varepsilon  \int_0^t \int_{\mathbb{R_+}} \bar{u}_x b^2  \,\mathrm{d}x \mathrm{d}\tau       \leq  \frac{5}{128} \int_0^t \int_{\mathbb{R_+}} \bar{u}_x b^2  \,\mathrm{d}x \mathrm{d}\tau ,
\end{equation}
where in the last inequality we have chosen
\begin{equation}\label{jiedianxishu-no.1}
   \frac{5}{2} \beta_1 \beta_3 \varepsilon \leq \frac{5}{128}, \quad \mathrm{i.e.} \quad  \beta_1 \beta_3 \varepsilon \leq \frac{1}{64}.
\end{equation}

By using the Cauchy inequality and the decay rate of smooth approximate rarefaction wave, we can deduce
\begin{align}
  \left| J_2 \right|
  & \leq  \int_0^t\int_{\mathbb{R_+}} \left| E + ub \right| \cdot \left| \frac{5}{4} \beta_3 \varepsilon \bar{u}_x b \right|  \,\mathrm{d}x \mathrm{d}\tau     \nonumber  \\[1mm]
  & \leq  \frac{1}{2} c_1 \int_0^t\int_{\mathbb{R_+}} \left( E + ub \right)^2  \,\mathrm{d}x \mathrm{d}\tau     + \frac{1}{2 c_1} \int_0^t\int_{\mathbb{R_+}} \frac{25}{16} \beta_3^2 \varepsilon^2 \bar{u}_x^2 b^2 \,\mathrm{d}x \mathrm{d}\tau            \nonumber  \\[1mm]
  & \leq  \frac{1}{2} c_1 \int_0^t\int_{\mathbb{R_+}} \left( E + ub \right)^2  \,\mathrm{d}x \mathrm{d}\tau      + \frac{25}{32 c_1} \beta_3^2 \varepsilon^2   \int_0^t\int_{\mathbb{R_+}} \| \bar{u}_x \|_{L^\infty_x}  \bar{u}_x b^2 \,\mathrm{d}x \mathrm{d}\tau            \nonumber  \\[1mm]
  & \leq  \frac{1}{2} c_1 \int_0^t\int_{\mathbb{R_+}} \left( E + ub \right)^2  \,\mathrm{d}x \mathrm{d}\tau      + \frac{25}{32 c_1} \beta_3^2 \varepsilon^2 \alpha  \int_0^t\int_{\mathbb{R_+}}  \bar{u}_x b^2 \,\mathrm{d}x \mathrm{d}\tau ,   \nonumber
\end{align}
where $ c_1 $ is a constant the same as that in \eqref{guji-jibennl-1}.
Then choosing $ \alpha $ suitably small leads to
\begin{equation}\label{guji-J2}
  \left| J_2 \right|  \leq    \frac{1}{2} c_1 \int_0^t\int_{\mathbb{R_+}} \left( E + ub \right)^2  \,\mathrm{d}x \mathrm{d}\tau       + \frac{3}{128} \int_0^t\int_{\mathbb{R_+}}  \bar{u}_x b^2 \,\mathrm{d}x \mathrm{d}\tau .
\end{equation}
Putting \eqref{guji-J1} and \eqref{guji-J2} into \eqref{guji-diancichang-xishubo}, we have
\begin{align}\label{yuxiang-diancichang-2}
\left|    \int_0^t \int_{\mathbb{R_+}}  \varepsilon E b \bar{u}_t  \,\mathrm{d}x\mathrm{d}\tau    \right|       \leq   \frac{1}{16}   \int_0^t \int_{\mathbb{R_+}} \bar{u}_x b^2  \,\mathrm{d}x \mathrm{d}\tau       + \frac{1}{2} c_1 \int_0^t\int_{\mathbb{R_+}} \left( E + ub \right)^2  \,\mathrm{d}x \mathrm{d}\tau   .
\end{align}

Next, recalling $  \left| \hat{u}(x,t) \right| < \frac{3}{2} \beta_1  $ and using the Cauchy inequality, the space integration term $ \int_{\mathbb{R_+}} \left| \varepsilon E \hat{u} b \right|  \,\mathrm{d}x  $ in \eqref{guji-jibennl-1} can be estimated as follows:
\begin{align}\label{guji-danjifen-E,b}
  \int_{\mathbb{R_+}} \left| \varepsilon E \hat{u} b \right|  \,\mathrm{d}x    \leq  \int_{\mathbb{R_+}} \varepsilon \left| E \cdot \frac{3}{2} \beta_1  b  \right|  \,\mathrm{d}x      \leq  \frac{1}{4} \int_{\mathbb{R_+}} \varepsilon E^2  \,\mathrm{d}x      + \frac{9}{4} \beta_1^2 \varepsilon \int_{\mathbb{R_+}} b^2 \,\mathrm{d}x     \leq    \frac{1}{4} \int_{\mathbb{R_+}} (\varepsilon E^2  +  b^2)  \,\mathrm{d}x  ,
\end{align}
where in the last inequality we have chosen
\begin{equation}\label{jiedianxishu-no.2}
   \beta_1^2 \varepsilon \leq \frac{1}{9}.
\end{equation}
Combining \eqref{yuxiang-diancichang-1}, \eqref{yuxiang-diancichang-2}, \eqref{guji-danjifen-E,b} and \eqref{guji-jibennl-1}, then choosing $\delta $ suitably small, we can obtain
\begin{align}\label{guji-jibennl-2}
&\quad   \int_{\mathbb{R_+}} \left( \rho\cdot \eta(x,t)    +  \frac{1}{4} \varepsilon E^2    + \frac{1}{4} b^2  \right) \,\mathrm{d}x         +     \frac{7}{16} \int_0^t \| \sqrt{\bar{u}_x}(\sqrt{\varepsilon}E, b) \|^2  \,\mathrm{d}\tau                          \nonumber  \\[2mm]
&\quad    + \frac{1}{2} \int_0^t  \sqrt{\varepsilon}E^2(0,\tau)  \,\mathrm{d}\tau       +  c_2 \int_0^t      \left[ \phi^2(0, \tau)       + \left\| \left(   \psi_x, \zeta_x,  E + \psi b + \hat{u} b   \right)  \right\|^2  \right]       \,\mathrm{d}\tau           \nonumber  \\[2mm]
& \leq  C \| (\phi_0, \psi_0, \zeta_0, E_0, b_0) \|^2        + C \delta  \int_0^t \| (\sqrt{\varepsilon} E_x, b_x) \|^2 \,\mathrm{d}\tau        + \int_0^t\int_{\mathbb{R_+}} Q \,\mathrm{d}x \mathrm{d}\tau .
\end{align}



Finally, for the terms $ \int_0^t\int_{\mathbb{R_+}} Q \,\mathrm{d}x \mathrm{d}\tau $, by performing some similar computations as in \cite{qinxiaohong-2011} and choosing $ \delta  $ and $ \alpha $ suitably small, we obtain
\begin{align}\label{guji-Q-jieguo}
& \quad   \int_0^t\int_{\mathbb{R_+}} Q \,\mathrm{d}x \mathrm{d}\tau      + c_3 \int_0^t\int_{\mathbb{R_+}}  \bar{u}_x \left| \left( \phi, \psi, \zeta \right) \right|^2   \,\mathrm{d}x \mathrm{d}\tau     \nonumber  \\[2mm]
& \leq   \frac{1}{4} c_2 \int_0^t  \| (\psi_x,\zeta_x) \|^2  \,\mathrm{d}\tau
  +  C \delta  \int_0^t \phi^2(0,\tau) \,\mathrm{d}\tau             + C (\delta  + \alpha^{\frac{1}{10}}) \left( 1 + \int_0^t  \| (\phi_x,\psi_x,\zeta_x) \|^2   \,\mathrm{d}\tau  \right) ,
\end{align}
where $ c_2 $ is a constant the same as that in \eqref{guji-jibennl-2}.
Putting \eqref{guji-Q-jieguo} into \eqref{guji-jibennl-2}, then taking $ \delta  $ and $ \alpha $ small enough, we can arrive at \eqref{jibennengliang}.

This completes the proof of Lemma \ref{lemmadijieenergy}.


\end{proof}


\subsection{High-order energy estimates}

\noindent In the following lemma, we will control the energy $\|(\sqrt{\varepsilon}E_x,b_x)\|^2$.


\begin{lemma}\label{gujiExbxlemma}
Suppose that the conditions in Proposition \ref{prop-1} hold.
Then for all $ 0 < t < T $, we have the following energy estimate
\begin{align}\label{gujiExbxjieguo}
  &\quad \| ( \sqrt{\varepsilon} E_x, b_x) \|^2        + \sqrt{\varepsilon} E^2(0,t)        +  \int_0^t  \| (E_x, b_x) \|^2  \,\mathrm{d}\tau      + \int_0^t \varepsilon^{\frac{3}{2}} E_\tau^2(0,\tau) \,\mathrm{d}\tau         \nonumber \\[1mm]
  & \leq   C \left( \| (\phi_0, \psi_0, \zeta_0) \|^2     + \| (E_0, b_0) \|^2_{H^1}  \right)       + C (\delta  + \alpha^{\frac{1}{10}}) \left( 1 +  \int_0^t \| \phi_x \|^2  \,\mathrm{d}\tau  \right) .
\end{align}
\end{lemma}

\begin{proof}[Proof]
Firstly, taking the derivative of $ \eqref{raodong-fuhebo}_4 $ with respect to $ x $ and multiplying it by $ E_x $, then integrating the resulting equality with respect to $ x $, we obtain
\begin{equation}\label{Exbx1}
\frac{\mathrm{d}}{\mathrm{d}t}  \int_{\mathbb{R_+}}   \frac{1}{2} \varepsilon E_x^2 \,\mathrm{d}x   - \int_{\mathbb{R_+}} b_{xx} E_x \,\mathrm{d}x        + \int_{\mathbb{R_+}} E^2_x \,\mathrm{d}x  =   - \int_{\mathbb{R_+}} (\psi b)_x E_x \,\mathrm{d}x        - \int_{\mathbb{R_+}} (\hat{u} b)_x E_x \,\mathrm{d}x   .
\end{equation}
Secondly, taking the derivative of $ \eqref{raodong-fuhebo}_5 $ with respect to $ x $ and multiplying $ b_x $, then integrating the resulting equality with respect to $ x $, we get
\begin{equation}\label{Exbx2}
\frac{\mathrm{d}}{\mathrm{d}t}  \int_{\mathbb{R_+}}  \frac{1}{2}  b^2_x \,\mathrm{d}x    - \int_{\mathbb{R_+}} E_{xx}b_x \,\mathrm{d}x  =0.
\end{equation}
Combining $ \eqref{Exbx1} $ and $ \eqref{Exbx2} $, we obtain
\begin{align}\label{Exbx3}
&\quad \frac{\mathrm{d}}{\mathrm{d}t}  \int_{\mathbb{R_+}}  \frac{1}{2} \left( \varepsilon E_x^2  +  b_x^2 \right)  \,\mathrm{d}x       - \int_{\mathbb{R_+}} \left( E_x b_x \right)_x  \,\mathrm{d}x          + \int_{\mathbb{R_+}} E_x^2 \,\mathrm{d}x       \nonumber  \\[1mm]
&=  - \int_{\mathbb{R_+}} (\psi b)_x E_x \,\mathrm{d}x     - \int_{\mathbb{R_+}} (\hat{u} b)_x E_x \,\mathrm{d}x.
\end{align}


Now we estimate the boundary term in \eqref{Exbx3}.
The boundary condition $\sqrt{\varepsilon}E(0,t)  =  b(0,t)$  implies
\begin{equation}\label{bdycons-Ebguanxi-2}
 \sqrt{\varepsilon}E_t(0,t) = b_t(0,t) .
\end{equation}
Moreover, taking $ x = 0 $ for $\eqref{raodong-fuhebo}_5$ leads to
\begin{equation}\label{bdycons-Ebguanxi-3}
  b_t(0,t) = E_x(0,t).
\end{equation}
Combining \eqref{bdycons-Ebguanxi-2} and \eqref{bdycons-Ebguanxi-3}, we obtain
\begin{equation}\label{bdycons-Ebguanxi-4}
   E_x(0,t)  = b_t(0,t) =  \sqrt{\varepsilon} E_t(0,t) .
\end{equation}
On the other hand, taking $ x = 0 $ for $\eqref{raodong-fuhebo}_4$, together with the boundary condition $\sqrt{\varepsilon}E(0,t)  =  b(0,t)$, we have
\begin{equation}\label{bdycons-Ebguanxi-5}
  b_x(0,t) =  \varepsilon E_t(0,t)  + E(0,t) + u_-b(0,t)
           =  \varepsilon E_t(0,t)  + \left( 1 + u_- \sqrt{\varepsilon} \right)  E(0,t).
\end{equation}
Using \eqref{bdycons-Ebguanxi-4} and \eqref{bdycons-Ebguanxi-5}, we can derive
\begin{align}\label{bdy-guji-B2-kaishi}
     &\quad  - \int_0^t\int_{\mathbb{R_+}}  \left( E_x b_x  \right)_x  \,\mathrm{d}x \mathrm{d}\tau
      = \int_0^t  \left( E_x b_x \right) (0,\tau)  \,\mathrm{d}\tau       \nonumber \\[2mm]
     &=  \int_0^t    \sqrt{\varepsilon} E_\tau(0,\tau) \cdot \left[ \varepsilon E_\tau(0,\tau)       + \left(  1  + u_- \sqrt{\varepsilon} \right) E(0,\tau)  \right]     \,\mathrm{d}\tau        \nonumber \\[2mm]
     &=  \int_0^t  \varepsilon^{\frac{3}{2}} E^2_\tau(0,\tau)  \,\mathrm{d}\tau            + \sqrt{\varepsilon} \left(  1 +  u_- \sqrt{\varepsilon} \right)  \int_0^t  (E_\tau E)(0,\tau)  \,\mathrm{d}\tau      \nonumber \\[2mm]
     &=   \int_0^t  \varepsilon^{\frac{3}{2}} E^2_\tau(0,\tau)  \,\mathrm{d}\tau        +   \frac{1}{2} \sqrt{\varepsilon} \left(  1 +  u_- \sqrt{\varepsilon} \right)  E^2(0,t)       - \frac{1}{2} \sqrt{\varepsilon} \left(  1 +  u_- \sqrt{\varepsilon} \right) E^2(0,0) .
\end{align}
Due to $ u_- < 0  $, we have
\begin{equation*}
  1 + u_-\sqrt{\varepsilon} = 1 - \left| u_- \right| \sqrt{\varepsilon}  \geq   1 - \beta_1 \sqrt{\varepsilon}.
\end{equation*}
By choosing
\begin{equation}\label{jiedianxishu-bdyguji-2}
  1 - \beta_1 \sqrt{\varepsilon}  \geq  \frac{1}{2}, \quad  \mathrm{i.e.}  \quad   \beta_1^2 \varepsilon \leq  \frac{1}{4} ,
\end{equation}
we can get
\begin{equation}\label{bdy-guji-B21}
   \frac{1}{2} \sqrt{\varepsilon} \left(  1 +  u_- \sqrt{\varepsilon} \right)  E^2(0,t)    \geq  \frac{1}{4} \sqrt{\varepsilon} E^2(0,t) .
\end{equation}
Furthermore, using the Sobolev inequality implies
\begin{equation*}
  \left| - \frac{1}{2} \sqrt{\varepsilon} \left(  1 +  u_- \sqrt{\varepsilon} \right) E^2(0,0)  \right|  \leq     \left|  \frac{1}{2} \sqrt{\varepsilon} \left(  1 +  u_- \sqrt{\varepsilon} \right) \right|  \cdot \| E_0 \|^2_{L^\infty}   \leq  C \| E_0 \|^2_{H^1} .
\end{equation*}
It follows that
\begin{equation}\label{bdy-guji-B22}
  - \frac{1}{2} \sqrt{\varepsilon} \left(  1 +  u_- \sqrt{\varepsilon} \right) E^2(0,0)   \geq  - \, C \| E_0 \|^2_{H^1} .
\end{equation}
By plugging \eqref{bdy-guji-B21} and \eqref{bdy-guji-B22} into \eqref{bdy-guji-B2-kaishi}, we can deduce the following boundary estimate
\begin{equation}\label{guji-boundary-eb2}
    - \int_0^t\int_{\mathbb{R_+}}  \left( E_x b_x  \right)_x  \,\mathrm{d}x \mathrm{d}\tau    \geq    \int_0^t  \varepsilon^{\frac{3}{2}} E^2_\tau(0,\tau)  \,\mathrm{d}\tau         + \frac{1}{4} \sqrt{\varepsilon} E^2(0,t)      - C \| E_0 \|^2_{H^1}  .
\end{equation}

Next, by using \eqref{xianyanjiashe}, $ \left| \hat{u}(x,t) \right| < \frac{3}{2} \beta_1 $, and the Cauchy inequality, we have
\begin{align}\label{Exbxyuxiang-1}
  & \quad - \int_{\mathbb{R_+}} (\psi b)_x E_x \,\mathrm{d}x     - \int_{\mathbb{R_+}} (\hat{u} b)_x E_x \,\mathrm{d}x     \nonumber   \\[1mm]
  & \leq  \int_{\mathbb{R_+}}  \left(  \left| \psi_x b E_x \right|     + \left| \psi b_x  E_x  \right|       +  \left| \hat{u} b_x E_x  \right|       +  \left| \hat{u}_x b E_x \right|  \right)     \,\mathrm{d}x    \nonumber \\[1mm]
  & \leq  C (\| b \|_{{L^{\infty}}}  + \| \psi \|_{{L^{\infty}}}   ) \cdot ( \| \psi_x \|^2  + \| E_x \|^2  +  \| b_x \|^2 )      \nonumber   \\[1mm]
  & \quad   +   \frac{1}{4} \int_{\mathbb{R_+}} E_x^2 \,\mathrm{d}x   +  \frac{9}{4}  \beta_1^2  \int_{\mathbb{R_+}} b_x^2 \,\mathrm{d}x         + \int_{\mathbb{R_+}} \left| \hat{u}_x b E_x  \right| \,\mathrm{d}x   \nonumber \\[1mm]
  &  \leq C \varepsilon_0 \| (\psi_x, E_x, b_x) \|^2       +   \frac{1}{4} \| E_x \|^2      +  \frac{9}{4} \beta_1^2   \| b_x \|^2          + \int_{\mathbb{R_+}} \left| \hat{u}_x b E_x  \right| \,\mathrm{d}x.
\end{align}
Moreover, using the Cauchy inequality, \eqref{guji-bianjieceng^2quan}, \eqref{guji-xishubo^2quan} and $\sqrt{\varepsilon}E(0,t)  =  b(0,t)$ leads to
\begin{align}\label{Exbxyuxiang-2}
  \int_{\mathbb{R_+}} \left| \hat{u}_x b E_x \right|  \,\mathrm{d}x
  & \leq  \frac{1}{8} \int_{\mathbb{R_+}} E^2_x \,\mathrm{d}x      +  C \int_{\mathbb{R_+}} \hat{u}^2_x b^2  \,\mathrm{d}x     \leq  \frac{1}{8} \| E_x \|^2       +  C \int_{\mathbb{R_+}} \tilde{u}^2_x b^2  \,\mathrm{d}x           +  C \int_{\mathbb{R_+}} \bar{u}^2_x b^2  \,\mathrm{d}x          \nonumber \\[2mm]
  & \leq  \frac{1}{8} \| E_x \|^2       + C \delta ^3  b^2(0,t)      +  C \delta ^2 \| b_x \|^2         +  C \varepsilon_0 \alpha^{\frac{1}{3}}  \left(  \| b_x \|^2   +  (1+t)^{-\frac{4}{3}}  \right)          \nonumber   \\[2mm]
  & \leq  \frac{1}{8} \| E_x \|^2       + C \delta ^3  \varepsilon E^2(0,t)      +  C \left( \delta ^2 + \varepsilon_0 \alpha^{\frac{1}{3}}  \right)            \| b_x \|^2         +  C \varepsilon_0 \alpha^{\frac{1}{3}} (1+t)^{-\frac{4}{3}} .
\end{align}
Thus putting \eqref{Exbxyuxiang-1}-\eqref{Exbxyuxiang-2} into \eqref{Exbx3},
and integrating the resulting inequality with respect to $ t $,
then employing \eqref{guji-boundary-eb2} and choosing $ \varepsilon_0$, $\alpha $ and $ \delta  $ suitably small gives
\begin{align}\label{Exguji-2}
  &\qquad  \| (\sqrt{\varepsilon}E_x, b_x ) \|^2       + \frac{1}{4} \sqrt{\varepsilon} E^2(0,t)       + \int_0^t \| E_x \|^2  \,\mathrm{d}\tau       + \int_0^t \varepsilon^{\frac{3}{2}} E_\tau^2(0,\tau)  \,\mathrm{d}\tau         \nonumber \\[2mm]
  & \leq  C \left( \| E_0 \|^2_{H^1} + \| b_{0x} \|^2   \right)     +  C \alpha^{\frac{1}{3}}    + C \varepsilon_0 \int_0^t \| \psi_x \|^2  \,\mathrm{d}\tau     + C\delta ^3 \int_0^t \varepsilon E^2(0,\tau) \,\mathrm{d}\tau        + 5 \beta_1^2 \int_0^t \| b_x \|^2  \,\mathrm{d}\tau .
\end{align}


Now we turn to estimate the term $ 5 \beta_1^2 \int_0^t \| b_x \|^2  \,\mathrm{d}\tau  $ in \eqref{Exguji-2}.
Multiplying $\eqref{raodong-fuhebo}_{4}$ by $-b_x$ and integrating the resulting equation with respect to $x$, together with $\eqref{raodong-fuhebo}_{5}$, we can deduce
\begin{align}
& \quad -\frac{\mathrm{d}}{\mathrm{d}t} \int_{\mathbb{R_+}} \varepsilon Eb_x \,\mathrm{d}x + \int_{\mathbb{R_+}}  b_x^2 \,\mathrm{d}x             \nonumber  \\[2mm]
& =   -\int_{\mathbb{R_+}} \varepsilon E b_{tx} \,\mathrm{d}x     + \int_{\mathbb{R_+}} (E + \psi b + \hat{u}b )b_x \,\mathrm{d}x          =    -\int_{\mathbb{R_+}} \varepsilon E E_{xx}  \,\mathrm{d}x      + \int_{\mathbb{R_+}} (E + \psi b + \hat{u}b)b_x \,\mathrm{d}x               \nonumber  \\[2mm]
& =   - \int_{\mathbb{R_+}} \varepsilon (E E_x)_x  \,\mathrm{d}x         +  \int_{\mathbb{R_+}} \varepsilon E^2_{x} \,\mathrm{d}x     + \int_{\mathbb{R_+}} (E + \psi b + \hat{u}b)b_x \,\mathrm{d}x               \nonumber\\[2mm]
& \leq  - \int_{\mathbb{R_+}} \varepsilon (E E_x)_x  \,\mathrm{d}x        + \int_{\mathbb{R_+}} \varepsilon E^2_{x} \,\mathrm{d}x         + \frac{1}{2} \int_{\mathbb{R_+}} (E + \psi b + \hat{u}b)^2 \,\mathrm{d}x       + \frac{1}{2} \int_{\mathbb{R_+}} b^2_x \,\mathrm{d}x , \nonumber
\end{align}
i.e.
\begin{equation}\label{gujibx-2}
- \frac{\mathrm{d}}{\mathrm{d}t} \int_{\mathbb{R_+}} 2\varepsilon Eb_x \,\mathrm{d}x      + \| b_x \|^2
\leq       - \int_{\mathbb{R_+}} 2\varepsilon (E E_x)_x  \,\mathrm{d}x      +  2 \varepsilon \| E_{x} \|^2       + \| E + \psi b + \hat{u}b \|^2.
\end{equation}
For the first term on the right-hand side of \eqref{gujibx-2}, using \eqref{bdycons-Ebguanxi-4} gives the following boundary estimate
\begin{align}\label{guji-boundary-eb3}
  - \int_0^t\int_{\mathbb{R_+}}  2 \varepsilon \left( E E_x  \right)_x  \,\mathrm{d}x \mathrm{d}\tau
  & =   \int_0^t 2 \varepsilon (E E_x)(0,\tau)  \,\mathrm{d}\tau     =    \int_0^t 2 \varepsilon^{\frac{3}{2}} (E E_\tau)(0,\tau) \,\mathrm{d}\tau         \nonumber \\[2mm]
  & =   \varepsilon^{\frac{3}{2}} E^2(0,t)      -  \varepsilon^{\frac{3}{2}} E^2(0,0)  \leq    \varepsilon^{\frac{3}{2}} E^2(0,t) .
\end{align}
Multiplying \eqref{gujibx-2} by $ 8 \beta_1^2 $ and integrating the resulting inequality with respect to $ t $, then adding it to \eqref{Exguji-2} and employing \eqref{guji-boundary-eb3} leads to
\begin{align}\label{gujibx-3}
   &\quad   \| ( \sqrt{\varepsilon} E_x, b_x) \|^2        + \left( \frac{1}{4}  -  8\beta_1^2 \varepsilon \right) \sqrt{\varepsilon} E^2(0,t)       + 3 \beta_1^2 \int_0^t \| b_x \|^2 \,\mathrm{d}\tau         \nonumber  \\[2mm]
   &\quad   +  (1 - 16 \beta_1^2 \varepsilon)  \int_0^t \| E_x \|^2 \,\mathrm{d}\tau           +  \int_0^t \varepsilon^{\frac{3}{2}} E^2_\tau (0,\tau) \,\mathrm{d}\tau                           \nonumber  \\[2mm]
   & \leq   C \left( \| E_0 \|^2_{H^1}   +  \| b_{0x} \|^2  \right)      + C \alpha^{\frac{1}{3}}           +  8 \beta_1^2  \int_0^t \| E + \psi b + \hat{u} b \|^2   \,\mathrm{d}\tau                   \nonumber  \\[2mm]
   & \quad    +  C \varepsilon_0  \int_0^t \| \psi_x \|^2  \,\mathrm{d}\tau       + C \delta ^3 \int_0^t \varepsilon E^2(0,\tau) \,\mathrm{d}\tau                + 16 \beta_1^2  \int_{\mathbb{R_+}}  \varepsilon \left| E b_x \right|  \,\mathrm{d}x  .
\end{align}
In order to absorb the bad term $ 16 \beta_1^2  \int_{\mathbb{R_+}}  \varepsilon \left| E b_x \right|  \,\mathrm{d}x $ and hold the time-space integrable good term $ \int_0^t \| E_x \|^2 \,\mathrm{d}\tau  $, we choose $ \beta_1^2 \varepsilon $ satisfying
\begin{equation}\label{jiedianxishu-no.3}
1  -  16 \beta_1^2 \varepsilon \geq \frac{1}{2} , \quad \mathrm{i.e.} \quad  \beta_1^2 \varepsilon  \leq  \frac{1}{32}.
\end{equation}
Hence,
\begin{equation}\label{gaojiedanjifenguji-1}
   16 \beta_1^2  \int_{\mathbb{R_+}} \varepsilon \left| E b_x \right|  \,\mathrm{d}x      \leq   16 \beta_1^2  \varepsilon  \int_{\mathbb{R_+}} \left( \frac{1}{2}  b^2_x   +  \frac{1}{2} E^2  \right)  \,\mathrm{d}x    \leq     \frac{1}{4} \| b_x \|^2       +  8 \beta_1^2 \| \sqrt{\varepsilon} E \|^2 .
\end{equation}
What's more, taking
\begin{equation}\label{jiedianxishu-no.4}
  \frac{1}{4}  -  8 \beta_1^2 \varepsilon \geq \frac{1}{8}, \quad \mathrm{i.e.} \quad   \beta_1^2 \varepsilon \leq  \frac{1}{64} ,
\end{equation}
so that we can transform the term $ \left( 1/4  -  8\beta_1^2 \varepsilon \right) \sqrt{\varepsilon} E^2(0,t) $ in \eqref{gujibx-3} into a good term.
Putting \eqref{gaojiedanjifenguji-1} into \eqref{gujibx-3} and using \eqref{jibennengliang},
then choosing $ \varepsilon_0,\,\delta  $ and $ \alpha $ suitably small,
we can conclude \eqref{gujiExbxjieguo}.

Recalling $ \beta_1 = \max\{\left| u_- \right|, \left| u_+ \right|\} $, $  \beta_2 = \max\{\theta_-, \theta_+ \} $ and $ \beta_3 = \beta_1 + \sqrt{R \gamma \beta_2} $,
from the discussion about $ \varepsilon $ in \eqref{jiedianxishu-bdyguji-1}, \eqref{jiedianxishu-no.1}, \eqref{jiedianxishu-no.2}, \eqref{jiedianxishu-bdyguji-2}, \eqref{jiedianxishu-no.3} and \eqref{jiedianxishu-no.4},
we can determine the constant $ \bar{C} $ in the following manner:
\begin{align}\label{jiedianxishu-jieguo}
\varepsilon
&  <  \min \left\{ \frac{1}{16 \beta_1^2}, \;\; \frac{1}{64 \beta_1 \beta_3}, \;\; \frac{1}{9 \beta_1^2}, \;\; \frac{1}{4 \beta_1^2}, \;\; \frac{1}{32 \beta_1^2}, \;\; \frac{1}{64 \beta_1^2}  \right\}          \nonumber  \\[2mm]
&  =  \frac{1}{64 \beta_1}   \min \left\{ \frac{1}{ \beta_3}, \;\; \frac{1}{\beta_1}  \right\}      =  \frac{1}{64 \beta_1 \beta_3 }  =   \frac{1}{64 \beta_1  \left(  \beta_1    +  \sqrt{\gamma R \beta_2} \right) }      =: \bar{C},
\end{align}
such that $ \varepsilon < \bar{C} $.
This completes the proof of Lemma $\ref{gujiExbxlemma}$.




\end{proof}


In the following two lemmas, we show the estimates for the first-order derivative $ (\phi_x, \psi_x, \zeta_x) $.

\begin{lemma}\label{gujiphixlemma}
Suppose that the conditions in Proposition \ref{prop-1} hold.
Then for all $ 0 < t < T $, we have the following energy estimate
\begin{equation}\label{gujiphixjieguo}
\| \phi_x \|^2  +  \int_0^t   \left(  \phi_{ x}^2(0,\tau)    + \| \phi_x \|^2  \right)  \,\mathrm{d}\tau
  \leq   C \left(  \| (\psi_0, \zeta_{0}) \|^2    +  \| (\phi_0, E_0, b_0) \|^2_{H^1}    \right)     +  C (\delta  + \alpha^{\frac{1}{10}}) .
\end{equation}

\end{lemma}

\begin{proof}

Motivated by the work of \cite{Kawashima1986-feidengshang} and \cite{qinxiaohong-2011}, firstly, differentiating  $\eqref{raodong-fuhebo}_{1}$ with respect to $x$ and multiplying the resulting equation by $ \frac{\phi_x}{\rho^3}$ yields
\begin{equation}\label{guji-phix1}
  \left(\frac{\phi_{x}^{2}}{2 \rho^{3}}\right)_{t}     + \left(\frac{u \phi_{x}^{2}}{2 \rho^{3}}\right)_{x}    -\hat{u}_{x} \frac{\phi_{x}^{2}}{\rho^{3}}      + \hat{\rho}_{x} \frac{\phi_{x} \psi_{x}}{\rho^{3}}     + \frac{\phi_{x} \psi_{x x}}{\rho^{2}}     =  f_{x} \frac{\phi_{x}}{\rho^{3}}  .
\end{equation}


To remove the high-order spatial derivative term $ \frac{\phi_x \psi_{xx}}{\rho^2}  $ of \eqref{guji-phix1},
we need to employ the equation $\eqref{raodong-fuhebo}_2$.
Multiplying $\eqref{raodong-fuhebo}_2$ by $ \frac{\phi_x}{\rho^2} $, after some elementary computations, we can get
\begin{align}\label{guji-varphix2}
&  \left(\frac{\phi_{x} \psi}{\rho}\right)_{t}-\left(\frac{\phi_{t} \psi}{\rho}+\frac{\hat{\rho}_{x} \psi^{2}}{\rho}\right)_{x}+\frac{\hat{\rho}_{x} \hat{u}_{x} \phi \psi}{\rho^{2}}+\frac{\hat{\rho}_{x x} \psi^{2}}{\rho}        + \frac{\left(\hat{\rho}_{x} \hat{u}-\bar{\rho} \bar{u}_{x}-\bar{\rho}_{x} \bar{u}\right) \psi \phi_{x}}{\rho^{2}}           - \psi_{x}^{2}      \nonumber   \\[2mm]
&   +  \frac{\left(2 \hat{\rho}_{x} \psi-\hat{u}_{x} \phi\right) \psi_{x}}{\rho}       + \frac{(p - \hat{p})_{x} \phi_{x}}{\rho^{2}}      - \frac{\hat{f} \psi_{x}}{\rho}      + \frac{\hat{\rho}_{x} \hat{f} \psi}{\rho^{2}}     =  \mu \frac{\phi_{x} \psi_{x x}}{\rho^{2}}   +   g \frac{\phi_{x}}{\rho^{2}}           -  (E + \psi b + \hat{u}b ) b \frac{\phi_x}{\rho^2}  .
\end{align}
$ \mu \times \eqref{guji-phix1}  + \eqref{guji-varphix2} $ gives
\begin{equation}\label{guji-Q4-1}
  \left( \frac{\mu \phi_{x}^{2}}{2 \rho^{3}}   + \frac{\phi_{x} \psi}{\rho} \right)_{t}     + \left(\frac{\mu u \phi_{x}^{2}}{2 \rho^{3}}-\frac{\phi_{t} \psi}{\rho}-\frac{\hat{\rho}_{x} \psi^{2}}{\rho}\right)_{x}       + \frac{ R \theta }{\rho^{2}} \phi_{x}^{2}   =   Q_{4}     -  (E + \psi b + \hat{u}b ) b \frac{\phi_x}{\rho^2}  ,
\end{equation}
where
\begin{align}
 Q_{4}
&  =  -\frac{2 \mu \hat{\rho}_{x} \phi_{x} \psi_{x}}{\rho^{3}}-\frac{\mu\left(\hat{u}_{x x} \phi+\hat{\rho}_{x x} \psi+\hat{f}_{x}\right) \phi_{x}}{\rho^{3}}-\frac{\hat{\rho}_{x} \hat{u}_{x} \psi \phi}{\rho^{2}}-\frac{\hat{\rho}_{x x} \psi^{2}}{\rho}      \nonumber  \\[2mm]
&\quad   - \frac{\left[\tilde{\rho}_{x} \hat{u}-\bar{\rho} \bar{u}_{x}+\bar{\rho}_{x}\left(\tilde{u}-u_{*}\right)\right] \psi \phi_{x}}{\rho^{2}}-2 \frac{\hat{\rho}_{x} \psi \psi_{x}}{\rho}+\psi_{x}^{2}      \nonumber  \\[2mm]
&\quad   +\frac{\hat{u}_{x} \phi \psi_{x}}{\rho}-\frac{\left(R \hat{\rho}_{x} \zeta +R \rho \zeta_{x}+R \hat{\theta}_{x} \phi\right) \phi_{x}}{\rho^{2}}+\frac{\hat{f} \psi_{x}}{\rho}-\frac{\hat{\rho}_{x} \hat{f} \psi}{\rho^{2}}      \nonumber  \\[2mm]
&\quad   -\left(\rho \hat{\rho} \hat{u}_{x} \psi-\hat{p}_{x} \phi+\mu \tilde{u}_{x x} \phi-\mu \hat{\rho} \bar{u}_{x x}+\rho \hat{g}\right) \frac{\phi_{x}}{\hat{\rho} \rho^{2}}      \nonumber  \\[2mm]
& \leq  \frac{R \theta}{2 \rho^{2}} \phi_{x}^{2}+C\left(\psi_{x}^{2}+\zeta_{x}^{2}\right)+C\left(\hat{f}^{2}+\hat{f}_{x}^{2}+\hat{g}^{2}+\bar{u}_{x x}^{2}\right)      \nonumber  \\[2mm]
&\quad   + C\left[\left(\tilde{u}_{x}^{2}+\left|\tilde{u}_{x x}\right|\right)+\left(\bar{u}_{x}^{2}+\left|\bar{u}_{x x}\right|\right)\right]\left(\phi^{2}+\psi^{2}+\zeta^{2}\right) .
\end{align}
Integrating \eqref{guji-Q4-1} with respect to $ x $ and $ t $, due to the good sign of $ u_- $ again, we obtain
\begin{align}\label{guji-phix-1}
&  \| \phi_x \|^2  +  \int_0^t   \left(  \phi_{ x}^2(0,\tau)    + \| \phi_x \|^2  \right)  \,\mathrm{d}\tau     \leq   C \| (\psi_0, \phi_{0 x}) \|^2        + \sum_{i=1}^3 K_i        \nonumber \\[2mm]
&  \qquad \quad  + C \left( \| \psi \|^2    + \int_0^t \| (\psi_x, \zeta_x, E + \psi b + \hat{u} b) \|^2  \,\mathrm{d}\tau   \right)  ,
\end{align}
where
\begin{align}\label{}
& K_1  : =  C \int_{0}^{t} \int_{\mathbb{R}_{+}}\left(\hat{f}^{2}+\hat{f}_{x}^{2}+\hat{g}^{2}+\bar{u}_{x x}^{2}\right) \,\mathrm{d} x \mathrm{d} \tau ,       \nonumber \\[2mm]
& K_2  : =  C \int_{0}^{t} \int_{\mathbb{R}_{+}}\left(\tilde{u}_{x}^{2}+\left|\tilde{u}_{x x}\right|\right)\left(\phi^{2}+\psi^{2}   + \zeta^{2}\right) \, \mathrm{d} x \mathrm{d} \tau  ,     \nonumber \\[2mm]
& K_3  : =  C \int_{0}^{t} \int_{\mathbb{R}_{+}}\left(\bar{u}_{x}^{2}+\left|\bar{u}_{x x}\right|\right)\left(\phi^{2}+\psi^{2} + \zeta^{2}\right)  \, \mathrm{d} x \mathrm{d} \tau  .   \nonumber
\end{align}


Recalling $ \bar{u}(0,t) = u_* $ and using the following elementary inequality on $ \bar{u} - u_* $:
\begin{equation}\label{budengshi-poincare2}
  \left| z(x,t) \right| = \left| z(0,t) + \int_0^x  z_y(y,t)  \, \mathrm{d} y  \right|  \leq  \left| z(0,t) \right| + x \| z_x \|_{ L^\infty_x } ,
\end{equation}
we can derive
\begin{align}\label{guji-K1}
K_1
& \leq  C \int_0^t\int_{\mathbb{R_+}} \left[  (\bar{u} - u_*)^2 \tilde{u}_x^2  +  (\tilde{u} - u_* )^2 \bar{u}_x^2  + \bar{u}_{xx}^2   \right]  \,\mathrm{d}x \mathrm{d}\tau    \nonumber \\[2mm]
& \leq  C \int_0^t\int_{\mathbb{R_+}}  \left[  x^2 \| \bar{u}_x \|^2_{L^\infty}  \tilde{u}_x^2      +  (\tilde{u} - u_* )^2 \bar{u}_x^2  + \bar{u}_{xx}^2   \right]  \,\mathrm{d}x \mathrm{d}\tau       \leq  C (\delta + \alpha ) .
\end{align}
Moreover, it is easy to deduce
\begin{align}
& K_2  \leq C \delta   \int_0^t \left(  \phi^2(0,\tau)    +    \| (\phi_x, \psi_x, \zeta_x) \|^2 \right)   \,\mathrm{d}\tau ,   \label{guji-K2}      \\[2mm]
& K_3 \leq C \alpha^{\frac{1}{3}} \left( 1  +  \int_0^t \| (\phi_x,\psi_x,\zeta_x) \|^2   \,\mathrm{d}\tau  \right) .      \label{guji-K3}
\end{align}
Plugging \eqref{guji-K1}-\eqref{guji-K3} into \eqref{guji-phix-1} then employing \eqref{jibennengliang} and \eqref{gujiExbxjieguo}, together with taking $ \delta  $ and $ \alpha $ small enough, we can conclude \eqref{gujiphixjieguo}.
This completes the proof of Lemma $\ref{gujiphixlemma}$.

\end{proof}


\begin{lemma}\label{gujipsixzetaxlemma}
Suppose that the conditions in Proposition \ref{prop-1} hold.
Then for all $ 0 < t < T $, we have the following energy estimate
\begin{equation}\label{guji-psixx,zetaxxjieguo}
 \| (\psi_x, \zeta_x) \|^2         + \int_0^t \| (\psi_{xx}, \zeta_{xx}) \|^2  \,\mathrm{d}\tau       \leq   C \left(   \| (\phi_0, \psi_0, \zeta_0, E_0, b_0) \|^2_{H^1}    + \delta   + \alpha^{\frac{1}{10}}  \right) .
\end{equation}

\end{lemma}

\begin{proof}

Motivated by the treatment can be done for the Navier-Stokes equations,
we multiply $\eqref{raodong-fuhebo}_2$ by $\frac{\psi_{xx}}{\rho}$ to obtain
\begin{align}\label{}
   \left( \frac{1}{2} \psi_x^2 \right)_t  - \left( \psi_t \psi_x  +  \frac{1}{2} u \psi_x^2  \right)_x      +  \mu \frac{\psi_{xx}^2 }{\rho}     &  =  - \frac{1}{2} u_x \psi_x^2    + (p - \hat{p})_x \frac{\psi_{xx}}{\rho}    - g \frac{\psi_{xx}}{\rho}       \nonumber \\[2mm]
   & \quad  + (E + \psi b  + \hat{u} b)b \frac{\psi_{xx}}{\rho}    .     \nonumber
\end{align}
Integrating the above equality over $  \mathbb{R}_+ \times [0,t] $ gives
\begin{align}\label{guji-psixx-chubu}
&  \| \psi_x \|^2      +  \int_0^t \| \psi_{xx} \|^2  \,\mathrm{d}\tau       + \int_0^t \| \sqrt{\bar{u}_x} \psi_x \|^2  \,\mathrm{d}\tau      \leq   C \| \psi_{0x} \|^2      +  C \int_0^t \psi_x^2 (0,\tau) \,\mathrm{d}\tau      + C \int_0^t\int_{\mathbb{R_+}} \left| \psi_x \right|^3   \,\mathrm{d}x \mathrm{d}\tau        \nonumber \\[2mm]
&\qquad\qquad\qquad\qquad  + C \int_0^t\int_{\mathbb{R_+}} | \tilde{u}_x | \psi_x^2  \,\mathrm{d}x \mathrm{d}\tau      + C \int_0^t\int_{\mathbb{R_+}} \left[ (p - \hat{p})_x^2  +  (E + \psi b + \hat{u} b)^2   +  g^2  \right]  \,\mathrm{d}x \mathrm{d}\tau  .
\end{align}
Utilizing the Sobolev inequality yields
\begin{align}\label{guji-zaxiang-1}
C \int_0^t \psi_x^2(0,\tau) \,\mathrm{d}\tau   \leq  C \int_0^t \| \psi_x \|^2_{L^\infty_x}  \,\mathrm{d}\tau    \leq   C  \int_0^t \| \psi_x \|   \| \psi_{xx} \|    \,\mathrm{d}\tau          \nonumber   \\[1mm]
\leq   \frac{1}{4} \int_0^t \| \psi_{xx} \|^2  \,\mathrm{d}\tau     + C \int_0^t \| \psi_x \|^2  \,\mathrm{d}\tau   ,
\end{align}
and
\begin{align}\label{guji-zaxiang-2}
&\quad  C\int_0^t\int_{\mathbb{R_+}}  \left| \psi_x \right|^3   \,\mathrm{d}x \mathrm{d}\tau     \leq    C \int_0^t \| \psi_x \|_{L^\infty_x} \| \psi_x \|^2   \,\mathrm{d}\tau     \leq  C \int_0^t  \| \psi_{xx} \|^{\frac{1}{2}} \| \psi_x \|^{\frac{5}{2}}  \,\mathrm{d}\tau       \nonumber  \\[1mm]
&  \leq   \frac{1}{4} \int_0^t \| \psi_{xx} \|^2  \,\mathrm{d}\tau      +  C  \sup_{0 \leq \tau \leq t} \| \psi_x \|^{\frac{4}{3}}   \int_0^t  \| \psi_x \|^2   \,\mathrm{d}\tau          \leq   \frac{1}{4} \int_0^t \| \psi_{xx} \|^2  \,\mathrm{d}\tau        + C \varepsilon_0^{\frac{4}{3}}  \int_0^t  \| \psi_x \|^2   \,\mathrm{d}\tau .
\end{align}
Similar to the treatment of \eqref{guji-K1}, by employing \eqref{guji-bianjieceng^2quan} and \eqref{guji-xishubo^2quan} in Lemma \ref{lemma-daiquanL2}, we have
\begin{align}\label{guji-zaxiang-3}
&\quad  C \int_0^t\int_{\mathbb{R_+}} | \tilde{u}_x | \psi_x^2  \,\mathrm{d}x \mathrm{d}\tau      + C \int_0^t\int_{\mathbb{R_+}} \left[ (p - \hat{p})_x^2  +  (E + \psi b + \hat{u} b)^2   +  g^2  \right]  \,\mathrm{d}x \mathrm{d}\tau     \nonumber \\[1mm]
& \leq    C \int_0^t\int_{\mathbb{R_+}}   \left[  | \tilde{u}_x | \psi_x^2       + ( \phi_x^2  +  \zeta_x^2  + \phi^2 \zeta_x^2  + \zeta^2 \phi_x^2 )     +  (\hat{\theta}_x^2 \phi^2  +   \hat{\rho}_x^2 \zeta^2 )    +  g^2       +  (E + \psi b + \hat{u} b)^2  \right]    \,\mathrm{d}x \mathrm{d}\tau       \nonumber \\[1mm]
& \leq  C \left\{    \int_0^t   \left[ \| (\phi_x,\psi_x,\zeta_x) \|^2     + \phi^2(0,\tau)     + \| E+ \psi b + \hat{u}b \|^2   \right]    \,\mathrm{d}\tau           +  \alpha^{\frac{1}{3}}  + \delta     \right\}   .
\end{align}
Plugging \eqref{guji-zaxiang-1}-\eqref{guji-zaxiang-3} into \eqref{guji-psixx-chubu},
then employing \eqref{jibennengliang}, \eqref{gujiExbxjieguo} and \eqref{gujiphixjieguo}, we can obtain
\begin{equation}\label{guji-psixx-jieguo}
\| \psi_x \|^2      +  \int_0^t \| \psi_{xx} \|^2  \,\mathrm{d}\tau       + \int_0^t \| \sqrt{\bar{u}_x} \psi_x \|^2  \,\mathrm{d}\tau    \leq   C \left( \| \zeta_0 \|^2    +  \| (\phi_0, \psi_0, E_0, b_0) \|^2_{H^1}    + \delta   + \alpha^{\frac{1}{10}}  \right)  .
\end{equation}


On the other hand, multiplying $\eqref{raodong-fuhebo}_3$ by $ \frac{\zeta_{xx}}{\rho} $, we can get
\begin{align}
  & \quad  \frac{R}{\gamma-1} \left[ \frac{1}{2} \left( \zeta_x^2 \right) _t  -  \left( \zeta_t \zeta_x  +  \frac{1}{2}  u \zeta_x^2 \right) _x    \right]     +  \frac{R}{2(\gamma-1)} \bar{u}_x \zeta_x^2     + \kappa \frac{\zeta_{xx}^2}{\rho}                    \nonumber \\[2mm]
  &   =   R \theta  \psi_x \zeta_{xx}   - \frac{R}{2(\gamma-1)} (\psi_x  +  \tilde{u}_x)\zeta_x^2      - \mu \psi_x^2 \frac{\zeta_{xx}}{\rho}    - h \frac{\zeta_{xx}}{\rho}     -  (E + \psi b + \hat{u} b)^2 \frac{\zeta_{xx}}{\rho}.    \nonumber
\end{align}
Integrating the above equality over $  \mathbb{R}_+ \times [0,t] $ gives
\begin{align}\label{guji-zetaxx-chubu}
 &\| \zeta_x \|^2        + \int_0^t \| \zeta_{xx} \|^2  \,\mathrm{d}\tau         +  \int_0^t \| \sqrt{\bar{u}_x} \zeta_x \|^2  \,\mathrm{d}\tau      \leq    C \| \zeta_{0x} \|^2    +  C \int_0^t  \zeta_x^2(0,\tau) \,\mathrm{d}\tau           \nonumber  \\[2mm]
 &\quad  + C \int_0^t\int_{\mathbb{R_+}}    \left\{ \left| \psi_x + \tilde{u}_x \right| \zeta_x^2 +   \left[ \left| \psi_x \right|  + \psi_x^2    + \left| h \right|      + (E + \psi b + \hat{u}b )^2  \right] \cdot \left| \zeta_{xx} \right|     \right\}    \,\mathrm{d}x \mathrm{d}\tau   .
\end{align}
Analogously, we can also obtain the following estimates:
\begin{align}\label{guji-zetaxx-1}
&\quad  C \int_0^t \zeta^2_x(0,\tau) \,\mathrm{d}\tau    + C \int_0^t\int_{\mathbb{R_+}}    \left\{ \left| \psi_x + \tilde{u}_x \right| \zeta_x^2 +   \left( \left| \psi_x \right|  + \psi_x^2    + \left| h \right|   \right)   \left| \zeta_{xx} \right|     \right\}    \,\mathrm{d}x \mathrm{d}\tau       \nonumber  \\[2mm]
& \leq      \frac{1}{4} \int_0^t \| \zeta_{xx} \|^2  \,\mathrm{d}\tau           +   C \left\{   \int_0^t  \left[ \| ( \phi_x, \psi_x, \zeta_x, \psi_{xx} ) \|^2    +  \phi^2(0,\tau) \right]   \,\mathrm{d}\tau        + \delta   +  \alpha^{\frac{1}{3}} \right\}    ,
\end{align}
and
\begin{align}\label{guji-zetaxx-2}
\int_0^t\int_{\mathbb{R_+}}   (E + \psi b + \hat{u}b )^2    \left| \zeta_{xx} \right|   \,\mathrm{d}x \mathrm{d}\tau   \leq  C \varepsilon_0 \int_0^t \| E + \psi b + \hat{u} b \| \| \zeta_{xx} \|   \,\mathrm{d}\tau     \nonumber \\[2mm]
\leq  \frac{1}{4} \int_0^t  \| \zeta_{xx} \|^2  \,\mathrm{d}\tau     + C \int_0^t  \| E + \psi b + \hat{u}b \|^2  \,\mathrm{d}\tau   .
\end{align}
Combining \eqref{guji-zetaxx-chubu}-\eqref{guji-zetaxx-2}, then employing \eqref{jibennengliang}, \eqref{gujiExbxjieguo}, \eqref{gujiphixjieguo} and \eqref{guji-psixx-jieguo} yields
\begin{equation}\label{guji-zetaxx-jieguo}
 \| \zeta_x \|^2        + \int_0^t \| \zeta_{xx} \|^2  \,\mathrm{d}\tau         +  \int_0^t \| \sqrt{\bar{u}_x} \zeta_x \|^2  \,\mathrm{d}\tau       \leq   C \left(   \| (\phi_0, \psi_0, \zeta_0, E_0, b_0) \|^2_{H^1}    + \delta   + \alpha^{\frac{1}{10}}  \right)   .
\end{equation}

Lemma \ref{gujipsixzetaxlemma} thus follows easily from \eqref{guji-psixx-jieguo} and \eqref{guji-zetaxx-jieguo}.

\end{proof}


\begin{proof}[Proof of Proposition \ref{prop-1}:]
We combine Lemma $\ref{lemmadijieenergy}$-Lemma $\ref{gujipsixzetaxlemma}$,
then choose $\varepsilon_0$, $\delta $ and $\alpha$ small enough to finish the proof of Proposition \ref{prop-1}.
\end{proof}


\section{Appendix: Derivation of 1-D models}

\noindent In this appendix, we will give a mathematical derivation of system \eqref{yuanfangcheng} for the completeness.
As a by-product, we also obtain some additional 1-D models.

NSM concerns the motion of conducting fluids (gases) in the electromagnetic fields with a very broad range of applications.
Since the dynamic motion of the fluid and the electromagnetic fields interact strongly on each other,
we must take into account the hydrodynamic and electrodynamic effects for the governing system.
The equations of NSM flows have the following form (\cite{Imai}, \cite{Kawashima1}, \cite{xux}):
\begin{equation}\label{chushifangcheng}
\left\{
\begin{aligned}
& \partial_{t} \rho+\operatorname{div}(\rho \boldsymbol{u})=0,   \\[1mm]
& \rho \left( \partial_{t}\boldsymbol{u} + \boldsymbol{u} \cdot \nabla \boldsymbol{u} \right)   +   \nabla p  =  \mu' \Delta\boldsymbol{u} + (\lambda+\mu') \nabla(\operatorname{div} \boldsymbol{u}) + \rho_e \boldsymbol{E} + \boldsymbol{J} \times \boldsymbol{B},  \\[1mm]
& \rho \frac{\partial e}{\partial \theta} (\partial_t \theta + \boldsymbol{u} \cdot \nabla \theta) + \theta \frac{\partial p}{\partial \theta} \operatorname{div}\boldsymbol{u} = \operatorname{div}(\kappa \nabla \theta)  +  \mathbb{N}(\boldsymbol{u})  +  (\boldsymbol{J}  -  \rho_e \boldsymbol{u}) \cdot (\boldsymbol{E} + \boldsymbol{u}\times\boldsymbol{B}),   \\[1mm]
& \varepsilon \partial_{t} \boldsymbol{E}  -  \frac{1}{\mu_0}\operatorname{curl}\boldsymbol{B}+\boldsymbol{J}=0,  \\[1mm]
& \partial_{t} \boldsymbol{B}  +  \operatorname{curl} \boldsymbol{E}=0,   \\[1mm]
&\partial_t \rho_e + \operatorname{div}\boldsymbol{J} = 0, \quad \varepsilon \operatorname{div}\boldsymbol{E} = \rho_e, \quad \operatorname{div}\boldsymbol{B} = 0,
\end{aligned}\right.
\end{equation}
where $(\boldsymbol{x}, t)\in\mathbb{R}^{3}\times\mathbb{R}_{+}$.
Here, $\rho(\boldsymbol{x},t)>0$ denotes the mass density,
$\boldsymbol{u}=\left(u_{1}, u_{2}, u_{3}\right)\in\mathbb{R}^{3}$ the fluid velocity,
$ \theta(\boldsymbol{x},t) > 0 $ the absolute temperature,
$\boldsymbol{E}=\left(E_{1}, E_{2}, E_{3}\right)\in\mathbb{R}^{3}$ the electric field,
$\boldsymbol{B}=\left(B_{1}, B_{2}, B_{3}\right)\in\mathbb{R}^{3}$ the magnetic field,
and $ \rho_e (\boldsymbol{x},t) $ the electric charge density.
The pressure $ p $ and the internal energy $ e $ are expressed by the equations of states for polytropic fluids: $ p = R \rho \theta$,\; $e = \frac{R}{\gamma-1} \theta $.
$ R >0 $ is the gas constant and $ \gamma>1 $ is the adiabatic exponent.
$ \mathbb{N}(\boldsymbol{u}) $ in $ \eqref{chushifangcheng}_3 $ denotes the viscous dissipation function:
$ \mathbb{N}(\boldsymbol{u}) = \sum_{i,j=1}^3 \frac{\mu'}{2} \left( \frac{\partial u_i}{\partial x_j} + \frac{\partial u_j}{\partial x_i} \right)^2 + \lambda \left( \operatorname{div}\boldsymbol{u} \right)^2 $.
$\mu'$ and $\lambda$ are the viscosity coefficients of the fluid which satisfy $\mu'>0$ and $2\mu'+3\lambda >0$.
The electric current density $\boldsymbol{J}$ can be expressed by Ohm's law: $ \boldsymbol{J} = \rho_e \boldsymbol{u} + \sigma(\boldsymbol E + \boldsymbol u \times \boldsymbol B) $.
$\sigma>0$ denotes the electic conductivity coefficient.
The heat conductivity coefficient $\kappa$ in $\eqref{chushifangcheng}_3$ and the magnetic permeability $\mu_0$ in $\eqref{chushifangcheng}_4$ are positive constants.
Finally, $\varepsilon>0$ in $\eqref{chushifangcheng}_4$ is the dielectric constant.


There is quite limited mathematical progress for the original nonlinear system since as pointed out by Kawashima in \cite{Kawashima1}, the system $\eqref{chushifangcheng}$ is neither symmetric hyperbolic nor strictly hyperbolic.
Because of this, the classical local well-posedness theorem (cf. \cite{Kato})
is invalid to the system $\eqref{chushifangcheng}$.
In addition, the tightly coupled hydrodynamic and electrodynamic effects produce strong nonlinearities, which leads to many difficulties.
Owing to the mathematically complicate structure of the original nonlinear system $\eqref{chushifangcheng}$,
some simplified models are derived according to the actual physical application.
As it was pointed out by Imai in \cite{Imai}, the assumption that the electric charge density $\rho_{e} \simeq 0$ is physically beneficial to the research of plasmas.
Here, we would mention that the quasi-neutrality assumption $\rho_{e} \simeq 0$ is different from the assumption of exact neutrality $\rho_{e} = 0$ since the latter would lead to the superfluous condition $\operatorname{div}\boldsymbol{E}=0 $.
According to this quasi-neutrality assumption, we can eliminate the terms involving $\rho_e$ in the system $ \eqref{chushifangcheng} $ and derive the following simplified system (cf. \cite{xux}):
\begin{equation}\label{NSM}
\left\{
\begin{aligned}
& \partial_{t} \rho+\operatorname{div}(\rho \boldsymbol{u})=0,   \\[1mm]
& \rho \left( \partial_{t}\boldsymbol{u} + \boldsymbol{u} \cdot \nabla \boldsymbol{u} \right)   +   \nabla p  =  \mu' \Delta\boldsymbol{u} + (\lambda+\mu') \nabla(\operatorname{div} \boldsymbol{u})  +  \boldsymbol{J} \times \boldsymbol{B},  \\[1mm]
& \frac{R}{\gamma-1} \rho (\partial_t \theta + \boldsymbol{u} \cdot \nabla \theta) + p \operatorname{div}\boldsymbol{u} = \operatorname{div}(\kappa \nabla \theta)  +  \mathbb{N}(\boldsymbol{u})  +  \boldsymbol{J} \cdot (\boldsymbol{E} + \boldsymbol{u}\times\boldsymbol{B}),  \\[1mm]
& \varepsilon \partial_{t} \boldsymbol{E}  -  \operatorname{curl}\boldsymbol{B}+\boldsymbol{J}=0,   \\[1mm]
& \partial_{t} \boldsymbol{B}  +  \operatorname{curl} \boldsymbol{E}=0, \quad \operatorname{div}\boldsymbol{B} = 0,
\end{aligned}\right.
\end{equation}
with the electric current density $\boldsymbol{J}  =  \boldsymbol E + \boldsymbol u \times \boldsymbol B $. Here we take $\mu_0 = \sigma = 1$ for simplicity.
The system \eqref{NSM} is usually called the Navier-Stokes-Maxwell equations (NSM) because it is obtained from the Navier-Stokes equations coupling with the Maxwell equations through the Lorentz force.

We shall derive the one-dimensional motion on a spatial domain.
Without loss of generality,
consider a three-dimensional NSM flow with spatial variables $ \boldsymbol{x} = ( x_1 , x_2 , x_3 )$ which is moving only in the longitudinal direction $ x_1 $ and uniform in the transverse directions $( x_2 , x_3 )$.
This means that all the quantities $ (\rho, \boldsymbol{u}, \theta, \boldsymbol{E}, \boldsymbol{B}) $ appearing in \eqref{NSM} are independent of the second and the third component of space variable $ (x_1 , x_2 , x_3 ) $.
According to the location of the dependent spatial variable, for example $ x_1 $ (below $x_1$ will be denoted by $x$),
we set $ \boldsymbol{u}:=(u,\; 0,\; 0) $ for the sake of the hydrodynamic feature of this one-dimensional flow.
If the dependent spatial variable takes $ x_2 $ (respectively, $ x_3 $), we can set $ \boldsymbol{u}:=(0,\; u,\; 0) $ (respectively, $ \boldsymbol{u}:=(0,\; 0,\; u) $) correspondingly.
And the discussion is analogous.
For this reason we have only to consider the following nine different cases.

\vspace{2mm}
$\bullet$ Case 1:\quad $\boldsymbol{u}:=(u,\; 0,\; 0), \quad  \boldsymbol{E}:=(0,\; 0,\; E),  \quad \boldsymbol{B}:=(0,\; b,\; 0). $

\vspace{1mm}
$\bullet$ Case 2:\quad $\boldsymbol{u}:=(u,\; 0,\; 0), \quad  \boldsymbol{E}:=(0,\; E,\; 0),  \quad \boldsymbol{B}:=(0,\; 0,\; \tilde{b}).  $

Denote $ \mu:= \lambda + 2\mu' $.
For Case 1, by employing direct calculations for \eqref{NSM}, we can deduce the following one-dimensional system:
\begin{equation}\label{system-1}
\left\{
\begin{aligned}
   &\rho_t + (\rho u)_x = 0,   \\
   &\rho (u_t + uu_x ) + p_x = \mu u_{xx} - (E + u b)b,   \\
   &\frac{R}{\gamma-1} \rho (\theta_t  + u \theta_x ) + p u_x = \mu u_x^2  +  \kappa \theta_{xx}  + (E + u b)^2,   \\
   &\varepsilon E_t  - b_x + E + u b =0,  \\
   &b_t - E_x =0 .
\end{aligned}
\right.
\end{equation}
For Case 2, we introduce a new dependent variable: $ b = - \,\tilde{b} $. Then the resulting system is the same as \eqref{system-1}.
We should mention that Fan et al. in \cite{fanhu} chose the component of $ \boldsymbol{u} $, $ \boldsymbol{E} $ and $ \boldsymbol{B} $ the same as that of Case 1 and first obtained system \eqref{system-1}.

\vspace{2mm}
$\bullet$ Case 3:\quad $\boldsymbol{u}:=(u,\; 0,\; 0), \quad  \boldsymbol{E}:=(0,\; 0,\; E), \quad \boldsymbol{B}:=(0,\; 0,\; b). $

\vspace{1mm}
$\bullet$ Case 4:\quad $\boldsymbol{u}:=(u,\; 0,\; 0), \quad  \boldsymbol{E}:=(0,\; E,\; 0), \quad \boldsymbol{B}:=(0,\; b,\; 0) $.

For these two cases, we can deduce the following system:
\begin{equation}\label{system-2}
\left\{
\begin{aligned}
   &\rho_t + (\rho u)_x = 0,    \\
   &\rho (u_t + uu_x ) + p_x = \mu u_{xx} - ub^2,    \\
   &\frac{R}{\gamma-1} \rho (\theta_t  + u \theta_x ) + p u_x = \mu u_x^2  +  \kappa \theta_{xx}    +  E^2    + (u b)^2,    \\
   &\varepsilon E_t   +  E  =  0,  \quad  E_x = 0,     \\
   &b_x - ub  = 0 ,    \quad   b_t =0.
\end{aligned}
\right.
\end{equation}
From $\eqref{system-2}_4$ and $\eqref{system-2}_5$, we can obtain
\begin{equation*}
E = E(t) = E(0) \,{\mathop{\mathrm{e}}}^{-\frac{t}{\varepsilon}},\quad   b = b(x) = b(0) \,{\mathop{\mathrm{e}}}^{\int_0^x u(y, 0)\,\mathrm{d}y}.
\end{equation*}


\vspace{2mm}
$\bullet$ Case 5:\quad $\boldsymbol{u}:=(u,\; 0,\; 0), \quad  \boldsymbol{E}:=(0,\; 0,\; E), \quad \boldsymbol{B}:=(b,\; 0,\; 0). $

\vspace{1mm}
$\bullet$ Case 6:\quad $\boldsymbol{u}:=(u,\; 0,\; 0), \quad  \boldsymbol{E}:=(0,\; E,\; 0), \quad \boldsymbol{B}:=(b,\; 0,\; 0) $.

For these two cases, we can deduce the following system:
\begin{equation}\label{system-3}
\left\{
\begin{aligned}
&\rho_t + (\rho u)_x = 0,     \\
&\rho (u_t + uu_x ) + p_x = \mu u_{xx} ,    \\
&\frac{R}{\gamma-1} \rho (\theta_t  + u \theta_x ) + p u_x = \mu u_x^2  +  \kappa \theta_{xx}    +  E^2 ,    \\
& E b=0,   \\
& \varepsilon E_{t}+E=0, \quad E_{x}=0,   \\
& b_{t}=0,    \quad   b_x=0 .
\end{aligned}
\right.
\end{equation}
$\eqref{system-3}_5$ and $\eqref{system-3}_6$ lead to
$$ E=E(t)=E(0) \,{\mathop{\mathrm{e}}}^{-\frac{t}{\varepsilon}}, \quad  b = \mathrm{constant}. $$
The equation $\eqref{system-3}_4$ permits at least one zero-solution, i.e.
$$ E=E(t)=E(0) \,{\mathop{\mathrm{e}}}^{-\frac{t}{\varepsilon}}, \quad  b = 0 ; $$
or
$$ E = 0, \quad  b = \mathrm{constant}. $$


\vspace{2mm}
$\bullet$ Case 7:\quad $\boldsymbol{u}:=(u,\; 0,\; 0), \quad  \boldsymbol{E}:=(E,\; 0,\; 0), \quad \boldsymbol{B}:=(0,\; b,\; 0) $.

\vspace{1mm}
$\bullet$ Case 8:\quad $\boldsymbol{u}:=(u,\; 0,\; 0), \quad  \boldsymbol{E}:=(E,\; 0,\; 0), \quad \boldsymbol{B}:=(0,\; 0,\; b). $

For these two cases, we can deduce the following system:
\begin{equation}\label{system-4}
\left\{
\begin{aligned}
&\rho_t + (\rho u)_x = 0,     \\
&\rho (u_t + uu_x ) + p_x = \mu u_{xx}   - u b^2 ,    \\
&\frac{R}{\gamma-1} \rho (\theta_t  + u \theta_x ) + p u_x = \mu u_x^2  +  \kappa \theta_{xx}    +  E^2    + (ub)^2 ,    \\
& E b=0,  \quad \varepsilon E_{t}+E=0,   \\
& b_x-ub=0, \quad  b_{t}=0.
\end{aligned}
\right.
\end{equation}
By a reasoning similar to the above, we obtain that
$$E = E(x,t) = E(x,0) \,{\mathop{\mathrm{e}}}^{ -\frac{t}{\varepsilon}},\quad b = 0 ; $$
or
$$ E = 0,  \quad   b = b(x) = b(0) \,{\mathop{\mathrm{e}}}^{\int_0^x  u(y,0) \,\mathrm{d}y} . $$


\vspace{2mm}
$\bullet$ Case 9:\quad $\boldsymbol{u}:=(u,\; 0,\; 0), \quad  \boldsymbol{E}:=(E,\; 0,\; 0),  \quad  \boldsymbol{B}:=(b,\; 0,\; 0). $

For this case, we can deduce the following system:
\begin{equation}\label{system-5}
\left\{\begin{aligned}
&\rho_t + (\rho u)_x = 0,     \\
&\rho (u_t + uu_x ) + p_x = \mu u_{xx}  ,    \\
&\frac{R}{\gamma-1} \rho (\theta_t  + u \theta_x ) + p u_x = \mu u_x^2  +  \kappa \theta_{xx}    +  E^2 ,    \\
& \varepsilon E_{t}+E=0,    \\
& b_{t}=0, \quad   b_x = 0.
\end{aligned}\right.
\end{equation}
From $\eqref{system-5}_4$ and $\eqref{system-5}_5$, we can obtain
$$E = E(x,t) = E(x,0) \,{\mathop{\mathrm{e}}}^{ -\frac{t}{\varepsilon}},\quad b = \mathrm{constant} . $$





As a result, for the above-mentioned cases, we obtain five 1-D compressible non-isentropic models: \eqref{system-1}-\eqref{system-5}.
We can see that system \eqref{system-1} (i.e., system \eqref{yuanfangcheng}) includes the electrodynamic effects into the dissipative structure of the hydrodynamic equations and turns out to be more complicated than that in the other four models.
Moreover, it may be interesting to note that the electromagnetic fields $ E $ and $ b $ in systems \eqref{system-2}-\eqref{system-5} can in principle be determined explicitly by the dielectric constant $ \varepsilon $, the boundary data for $ b $ at $ x = 0 $, and the initial data for $ E $ and $u$.

Further discussion about these five one-dimensional models in mathematics and physics would be meaningful and interesting.



\section*{Acknowledgments}


\noindent The research was supported by the National Natural Science Foundation of China $\#$11771150, 11831003 and Guangdong Basic and Applied Basic Research Foundation $\#$2020B1515310015. 


\end{spacing}
\end{document}